\documentclass[12pt, reqno]{amsart}
\usepackage{amssymb,amsthm,amsfonts,amsmath, amscd}

\usepackage{hyperref}
\usepackage{mathrsfs}
\usepackage{color}

\newcommand\Y{\mathbb Y}
\newcommand\Z{\mathbb Z}
\newcommand\C{\mathbb C}
\newcommand\R{\mathbb R}
\newcommand\T{\mathbb T}

\newcommand\E{\mathbb E}

\renewcommand\Y{\mathbb Y}

\newcommand\al{\alpha}
\newcommand\be{\beta}
\newcommand\ga{\gamma}
\newcommand\Ga{\Gamma}

\newcommand\de{\delta}

\newcommand\la{\lambda}

\newcommand\epsi{\varepsilon}
\newcommand\om{\omega}
\newcommand\Om{\Omega}

\newcommand\X{\mathfrak X}
\newcommand\su{\mathfrak{su}(1,1)}
\newcommand\sltwo{\mathfrak{sl}(2,\C)}

\newcommand\Sym{\operatorname{Sym}}
\newcommand\Ran{\operatorname{Ran}}
\newcommand\const{\operatorname{const}}
\newcommand\sgn{\operatorname{sgn}}

\newcommand\Span{\operatorname{Span}}
\newcommand\SSpan{\overline{\Span}}
\newcommand\Cyl{\operatorname{Cyl}}
\newcommand\GAMMA{{\operatorname{gamma}}}
\renewcommand\Re{\operatorname{Re}}

\newcommand\tr{\operatorname{Tr}}
\newcommand\detreg{\operatorname{det}_2}

\newcommand\wt{\widetilde}
\newcommand\wh{\widehat}

\newcommand\one{\mathbf1}
\newcommand\K{\mathbf K}

\renewcommand\P{\mathcal P}

\newcommand\z{{(z,z')}}

\newcommand\zzz{{(z+1,z'+1)}}

\newcommand\PPsi{\overline\Psi}

\newcommand\ccdot{\,\cdot\,}
\newcommand\F{\mathcal F}

\newtheorem{theorem}{Theorem}[section]
\newtheorem{proposition}[theorem] {Proposition}
\newtheorem{corollary}[theorem]{Corollary}
\newtheorem{lemma}[theorem]{Lemma}

\theoremstyle{definition}
\newtheorem{definition}[theorem]{Definition}
\newtheorem{remark}[theorem]{Remark}
\newtheorem{example}[theorem]{Example}
\newtheorem{question}[theorem]{Question}

\numberwithin{equation}{section}

\begin{document}

\title[]
{A hierarchy of Palm measures for determinantal point processes with gamma kernels}

\author{Alexander I. Bufetov and Grigori Olshanski}
\address{A. Bufetov: Aix-Marseille Universit\'e, Centrale Marseille, CNRS, Institut de Math\'e\-matiques de Marseille, UMR7373, 39 Rue F. Joliot Curie 13453, Marseille, France; Steklov Mathematical Institute of RAS, 8 Gubkina 119991, Moscow, Russia; Institute for Information Transmission Problems, Bolshoy Karetny 19, Moscow 127051, Russia. \newline \indent \emph{Email}: {\tt alexander.bufetov@univ-amu.fr}}
\address{G. Olshanski: Institute for Information Transmission Problems, Bolshoy Karetny 19, Moscow 127051, Russia; Skolkovo Institute of Science and Technology, Moscow, Russia; Department of Mathematics, National Research University Higher School of Economics, Moscow, Russia.
\newline \indent \emph{Email}: {\tt olsh2007@gmail.com}}

\date{}

\begin{abstract}
The gamma kernels are a family of projection kernels $K^{(z,z')}=K^{(z,z')}(x,y)$ on a doubly infinite $1$-dimensional lattice. They are expressed through Euler's gamma function and depend on two continuous parameters $z,z'$. The gamma kernels initially arose from a model of random partitions via a limit transition. On the other hand, these kernels are closely related to unitarizable representations of the Lie algebra $\mathfrak{su}(1,1)$. Every gamma kernel $K^{(z,z')}$ serves as a correlation kernel for a determinantal measure  $M^{(z,z')}$, which lives on the space of infinite point configurations on the lattice.  

We examine chains of kernels of the form
$$
\ldots, K^{(z-1,z'-1)}, \; K^{(z,z')},\; K^{(z+1,z'+1)}, \dots,
$$
and establish the following hierarchical relations inside any such chain: 

Given $(z,z')$, the kernel $K^{(z,z')}$ is a one-dimensional perturbation of (a twisting of) the kernel $K^{(z+1,z'+1)}$, and  the one-point Palm distributions for the measure $M^{(z,z')}$ are absolutely continuous with respect to $M^{(z+1,z'+1)}$.  

We also explicitly compute the corresponding Radon-Nikod\'ym derivatives and show that they are given by certain normalized multiplicative functionals.
\end{abstract}

\maketitle

\tableofcontents

\section{Introduction}\label{sect1}

\subsection{Determinantal measures}\label{sect1.1}

Let $\X$ be a countable set and $\Om:=\{0,1\}^\X$ denote the set of subsets of $\X$; it is a compact space in the product topology. Let $\P(\Om)$ denote the space of probability Borel measures on $\Om$. Given an $n$-point subset $\{x_1,\dots,x_n\}\subset\X$,  let $\Cyl(x_1,\dots,x_n)$  denote the cylinder subset of $\Om$ consisting of those $\om\in\Om$ that contain all $x_i$'s. Any measure $M\in\P(\Om)$ is uniquely determined by its \emph{correlation functions} $\rho_1,\rho_2,\dots$, where 
\begin{equation}\label{eq1.A}
\rho_n(x_1,\dots,x_n):=M(\Cyl(x_1,\dots,x_n)), \qquad \text{$x_1,\dots,x_n\in\X$, pairwise distinct.}
\end{equation}

A measure $M\in\P(\Om)$ is said to be a \emph{determinantal measure} if there exists a complex function $K(x,y)$ on $\X\times\X$ such that 
$$
\rho_n(x_1,\dots,x_n)=\det[K(x_i,x_j]_{i,j=1}^n, \quad n=1,2,\dots\,.
$$
Any such function (typically, it is not unique) is called  \emph{a correlation kernel} of $M$. Obviously, $M$ is uniquely determined by any its correlation kernel. 

Consider the complex Hilbert space $\ell^2(\X)$. It has a distinguished orthonormal basis $\{e_x: x\in\X\}$ formed by the delta-functions. For a bounded operator $K$ on $\ell^2(\X)$ we set $K(x,y):=(K e_y,e_x)$ and call $K(x,y)$ the \emph{matrix} of $K$.  It is well known that for any selfadjoint operator $K$ such that $0\le K\le1$, the matrix $K(x,y)$ serves as a correlation kernel of a determinantal measure $M\in\P(\Om)$. 

In particular, if $K$ is a selfadjoint projection operator, that is, the operator of orthogonal projection onto a subspace $L\subseteq\ell^2(\X)$, then $K$ gives rise to a determinantal measure. Let us denote it by $M[L]$ and call it a \emph{projection kernel measure}.

For more detail about determinantal measures, see the surveys \cite{BHKPV}, \cite{Bor}, \cite{Ly-2003}, \cite{Ly-2014}, \cite{Soshnikov}. 

\subsection{The gamma kernel measures}

Let $\Y$ denote the set of partitions, which we identify with their Young diagrams. Let us identify $\X$ with the set $\Z':=\Z+\tfrac12$ of half-integers, so that $\Om=\{0,1\}^{\Z'}$. We embed $\Y$ into $\Om$ by making use of the map 
\begin{equation}\label{eq1.B}
\Y\ni \la \mapsto \om(\la):=\{\la_i-i+\tfrac12: i=1,2,\dots\}
\end{equation}
and regard $\Om$ as a compactification of the discrete set $\Y$. 

In connection with the problem of harmonic analysis on the infinite symmetric group, the paper \cite{BO-2000} introduced a three-parameter set of determinantal measures on $\Y\subset\Om$ called the (mixed) \emph{z-measures}. We denote them by $M^\z_\xi$, where $\z$ is a pair of complex parameters (subject to some constraints specified below) and $\xi$ is a real parameter such that $0<\xi<1$.  The measures $M^\z_\xi$ are a special (and in many respects distinguished) example of \emph{Schur measures}. 

As shown in \cite{BO-2005}, for $\z$ fixed, there exists a weak limit
$$
M^\z:=\lim_{\xi\to1}M^\z_\xi\in\P(\Om).
$$
The limit measures $M^\z$ are our object of study. Here are some of their properties.

\smallskip

$\bullet$ Unlike the z-measures, the measures $M^\z$ are no longer supported by $\Y\subset\Om$.

$\bullet$ Each $M^\z$ is a determinantal measure. It admits a correlation kernel $K^\z(x,y)$, which can be written in the so-called \emph{integrable form}, meaning that
\begin{equation}\label{eq1.C}
K^\z(x,y)=\frac{\mathcal A(x)\mathcal B(y)-\mathcal B(x)\mathcal A(y)}{x-y},\quad x,y\in\Z',
\end{equation}
with a suitable resolution of singularity on the diagonal $x=y$. Here $\mathcal A(x)$ and $\mathcal B(x)$ are certain functions on $\Z'$ that depend on $\z$ and are expressed through the Euler gamma function. For this reason $K^\z(x,y)$ is called the \emph{gamma kernel} and $M^\z$ is called the \emph{gamma kernel measure}. 

$\bullet$ The gamma kernel corresponds to a selfadjoint projection operator $\K^\z$ on the Hilbert space $\ell^2(\Z')$, so that $M^\z$ belongs to the class of projection kernel measures. 

$\bullet$ The projections $\K^\z$ are closely related to unitarizable representations of the Lie algebra $\su$ of the principal and complementary series.  

\smallskip

\subsection{Link with representations of $\su$}

Let us explain the last point in more detail. First, let us specify the constraints on $\z$. We suppose that $\z$ satisfies one of the following two conditions: either $z'=\bar z\in\C\setminus\Z$ or there exists $\ell\in\Z$ such that both $z$ and $z'$ are real and contained in $(\ell,\ell+1)$.  One can write down a family $\{S^\z\}$ of irreducible representations of the Lie algebra $\sltwo$, which are realized in the subspace of finitely supported functions in $\ell^2(\Z')$ and are unitarizable with respect to the noncompact real form $\su\subset\sltwo$.  Set $\mathbf X:=\begin{bmatrix} -1 &1\\-1&1\end{bmatrix}$; it  is a nilpotent element of $\sltwo$. The corresponding operator $S^\z(\mathbf X)$ is a essentially selfadjoint, so  that its closure $\overline{S^\z(\mathbf X)}$ is a selfadjoint operator. The spectrum of  $\overline{S^\z(\mathbf X)}$ is purely continuous, filling the whole real axis, and it turns out that $\K^\z$ coincides with the spectral projection on the positive part of the spectrum. 

The representations $S^\z$ with $z'=\bar z$ constitute the principal series and those with $z,z'\in(\ell,\ell+1)$ form the complementary series. Below we extend this terminology to the projections $\K^\z$, too. Note that for any $m\in\Z\setminus\{0\}$,  the representations $S^\z$ and $S^{(z+m, z'+m)}$ are equivalent, but  $\K^\z\ne\K^{(z+m,z'+m)}$.  

\smallskip

\subsection{An inductive relation between the projections $\K^\z$.}

In what follows we denote by the symbol $\Ran(\ccdot)$ the range of a linear operator. 
We start by introducing a special basis for the  range $\Ran(\K^\z)\subset\ell^2(\Z')$ of the projection $\K^\z$. Consider the following functions on $\Z'$
\begin{multline}\label{eq3.g-prelim}
g^\z_m(x)
=\frac{\sin(\pi z')\Ga(1+z-z')}{\pi}\, \frac{\Ga(x+z+\tfrac12)}{\sqrt{\Ga(x+z+\tfrac12)\Ga(x+z'+\tfrac12)}} \,\frac{\Ga(x+z'+m+\tfrac12)}{\Ga(x+z+m+\tfrac32)}. 
\end{multline}
For these functions we establish, cf.   Corollary \ref{cor3.D},  the key relations
\begin{equation}\label{eq3.r2-prelim}
K^\z(x,y)=\sum_{m=0}^\infty g^\z_m(x)g^{(z',z)}_m(y),
\end{equation}
\begin{multline}\label{eq3.r4-prelim}
K^\z(x,y)=\sum_{i=0}^{m-1} g^\z_i(x)g^{(z',z)}_i(y)\\
+ \frac{A^\z(x) A^{(z+m,z'+m)}(y)}{A^{(z+m,z'+m)}(x) A^\z(y)}\, K^{(z+m,z'+m)}(x,y), \qquad m=1,2,\dots.
\end{multline}
Let $V^\z_m$ stand for the operator  of multiplication by the function $A^\z(x)/A^{(z+m,z'+m)}(x)$.
The relation \eqref{eq3.r4-prelim} implies, see Theorem \ref{thm3.B} below, that the  
space 
$$
\Ran(V^\z_m \K^{(z+m,z'+m)} (V^\z_m)^{-1})
$$ 
is  contained in the space $\Ran(\K^\z)$ and has codimension $m$.

The relation \eqref{eq3.r2-prelim} and Theorem \ref{thm3.B} hold for all admissible values of the parameters $(z,z')$ but have a different interpretation in the case of the principal and the complementary series. 
In the case of the principal series, when $z=\overline{z'}$, the functions $g^{(z, \overline{z})}_m$ form an orthonormal basis for the subspace $\Ran(\K^\z)$, see Theorem \ref{thm3.A} below. Furthermore, 
the operator  $V^\z_m$ is unitary and the direct sum in Theorem \ref{thm3.B} is an orthogonal direct sum: the subspace $\Ran(\K^\z)$ is thus represented as an explicit rank $m$ perturbation of the subspace $\Ran \K^{(z+m,z'+m)}$ twisted by a unitary operator. For the principal series we also explain a relationship of the subspaces $\Ran\K^{(z,\bar z)}$ to the classical \emph{Beurling theorem} \cite{Helson}. 

In the case of the complementary series, the direct sum is no longer orthogonal, and the twist is no longer unitary. It is nevertheless true, see Theorem \ref{thm3.C} below,  that  the subspace $\Ran(\K^\z)$ is the closed linear span of the functions $g^\z_0, g^\z_1, g^\z_2,\dots$\,. The relation \ref{eq3.r2-prelim} implies that the biorthogonal family is precisely the family $g^{(z',z)}_0, g^{(z',z)}_1, g^{(z',z)}_2, \dots$.

\subsection{Reduced Palm measures and the main result}
We are now ready to give an informal description of the main result of this paper. 

In the case of a point process on a discrete space $\X$, the Palm measure is simply the conditional measure subject to the condition that our process have a particle at a given site $p\in\X$. For notational convenience, this particle, on whose existence one conditions, is then removed. The resulting conditional measure is called the \emph{reduced Palm measure}. Let us denote it by $M(p)$, where $M\in\P(\{0,1\}^\X)$ stands for the distribution of the initial point process. We are interested in  the reduced Palm measure $M^\z(p)$, where $M^\z$ is a gamma kernel measure and $p\in\Z'$. Note that $M^\z(p)$ and $M^\z$ are mutually singular, see  \cite[Proposition 4.3]{Buf-rigid}.

The main result of our paper, Theorem \ref{thm6.A},  states that $M^\z(p)$  is absolutely continuous with respect to  $M^{(z+1, z'+1)}$ and the Radon-Nikod\'ym derivative $M^\z(p)/M^\z$ can be computed explicitly. 

In more detail,  for $p\in \mathbb Z'$, introduce a function $a^\z_p$ on $\mathbb Z'$ by the formula
\begin{equation}\label{azp-def}
a^\z_p(x):=\dfrac{(x-p)^2}{(x+z+\tfrac12)(x+z'+\tfrac12)}, \qquad x\in\Z',
\end{equation}
and write for $\om\in\Om$
\begin{equation}\label{azp-funct}
\check \Psi_{p;z,z'}(\omega):=\begin{cases}  \prod\limits_{x\in \omega, x>0} a^\z_p(x)\cdot  \prod\limits_{y\notin \omega, y<0, y\ne p} \left(a^\z_p\right)^{-1}(y),  & p\notin\om, \\
0, & p\in\om.
\end{cases}
\end{equation}
One can prove that both infinite products in \eqref{azp-funct}
converge $M^{(z+1,z'+1)}$-almost surely. Theorem \ref{thm6.A}  states (in an equivalent formulation, see Remark \ref{rem6.C}) that $M^\z(p)/M^\z$ coincides with 
$\check \Psi_{p;z,z'}$ up to a constant factor.

\subsection{Acknowledgements.} A. Bufetov has received funding from the European Research Council (ERC) under the European Union’s Horizon 2020 research and innovation programme under grant agreement No 647133 (ICHAOS) and was also supported by the Russian 
Foundation for  Basic Research, grant 18-31-20031. A.B. is deeply grateful to the warm hospitality of  the Pontificia Universidad Cat\'olica de Valpara\'iso, where a part of this work was written.

\section{Preliminaries on the gamma kernels $K^\z(x,y)$}

\subsection{The z-measures $M^\z_\xi$}

Recall that $\Y$ denotes the set of partitions, which we identify with their Young diagrams. Given a partition $\la\in\Y$, we set
$$
|\la|:=\la_1+\la_2+\dots;
$$
equivalently, $|\la|$ is the number of boxes of the Young diagram $\la$. Next,
we set
$$
(z)_\la:=\prod_{\Box\in\la}(z+c(\la)), \qquad z\in\C,
$$
where the product is taken over the set of boxes of $\la$ and $c(\Box)$ is the
{\it content\/} of a box $\Box$ (that is, $c(\Box)=j-i$, where $i$ and $j$ are
the row and column numbers of $\Box$). Finally, by $\dim\la$ we denote the
number of standard tableaux of shape $\la$; this number equals the
dimension of the irreducible representation of the symmetric group of degree
$|\la|$ indexed by $\la$.

\begin{definition}\label{def2.A}
We say that a pair $\z$ of complex or real parameters is \emph{admissible} if one of the following two conditions holds: 
\smallskip

$\bullet$ $z\in\C\setminus\Z$ and  $z'=\bar z$;

$\bullet$ $z$ and $z'$ are distinct real numbers contained in an open interval of the form $(\ell, \ell+1)$ for some $\ell\in\Z$.

We say that $\z$ is in the \emph{principal series} in the first case, and in the \emph{complementary series}, in the second case. 
\end{definition}

\begin{definition}\label{def2.B}
Let $(z,z')$ be an admissible pair of parameters and $\xi\in(0,1)$ be an extra parameter. With such a triple $(z,z',\xi)$ we associate a measure on $\Y$ called the (mixed) {\it
z-measure\/} and denoted by $M^\z_\xi$:
\begin{equation}\label{eq2.A}
M^\z_\xi(\la)=(1-\xi)^{zz'}\xi^{|\la|}(z)_\la(z')_\la\left(\frac{\dim\la}{|\la|!}\right)^2,
\qquad \la\in\Y.
\end{equation}
Note that $M^\z_\xi=M^{(z',z)}_\xi$. 
\end{definition}

\smallskip

Due to the admissibility assumption one has $zz'>0$ and the quantities
$$
(z)_\la(z')_\la=\prod_{\Box\in\la}(z+c(\Box))(z'+c(\Box))
$$
are real and strictly positive for all $\la$.  It follows that $M^\z_\xi(\la)$ is real
and strictly positive for every $\la\in\Y$. It is known that $M^\z_\xi$ is a
probability measure. 

The measures $M^\z_\xi$ were introduced in \cite{BO-2000} as `mixtures' of certain probability measures on the finite sets $\Y_n$ (partitions of $n=0,1,2,\dots$). This is a special (and in many respects distinguished) example of Schur measures \cite{Ok-2001a}. For more detail see \cite{BO-2005}, \cite{BO-2006a}, \cite{BO-2006b}. Note that those papers used the different notation $M_{z,z',\xi}$. 

 As explained in Section \ref{sect1}, we may treat $M^\z_\xi$ as a probability measure on $\Om=\{0,1\}^{\Z'}$. It is known that  
every mixed z-measure $M^\z_\xi$ is a determinantal measure. This result was first established in \cite[Proposition 3.1]{BO-2000}. Another proof is based on the fact that the measures $M^\z_\xi$ are a specal case of Schur measures, see \cite{Ok-2001a}, \cite{BOk-2000}. Yet other proofs are given in \cite{Ok-2001b}, \cite{BO-2006b}, \cite{BOS-2006}. 
Note that the methods of \cite{BO-2006b} and \cite{BOk-2000} lead to distinct correlation kernels (they differ by a `gauge transformation' that does not affect the correlation functions). 

\subsection{The limit measures $M^\z$}\label{sect2.2}
Fix an arbitrary admissible pair $\z$. Since the space $\Om$ is compact, the family $\{M^\z_\xi: 0<\xi<1\}$ is tight. Then it is natural to ask what happens as $\xi$ tends to one of the two boundary points of the interval $(0,1)$. It is readily seen that as $\xi\to0$,  the measures $M^\z_\xi$ weakly converge to the delta-measure at the point $\om(\varnothing)\in\Om$ corresponding to the empty diagram $\varnothing\in\Y$ (this point represents the subset $\Z'_{<0}\subset\Z'$). As $\xi\to1$, the picture is nontrivial. Namely, then the measures $M^\z_\xi$ weakly converge to a probability measure on $\Om$ (see \cite[Theorem 2.3]{BO-2005}). Unlike the pre-limit measures which are supported by $\Y\subset\Om$, the limit measure does not charge $\Y$ at all.  We denote this limit measure by $M^\z$ (in \cite{BO-2005} it was denoted by $\underline{P}^\GAMMA_{z,z'}$). 

\subsection{The gamma kernel $K^\z(x,y)$}

As above, we suppose that $\z$ is admissible. 
In \cite{BO-2005} it was shown that $M^\z$ is a determinantal measure and possesses a correlation kernel which is expressed through Euler's gamma function and called the \emph{gamma kernel}. For this reason $M^\z$ is called the \emph{gamma kernel measure}.

In \cite{BO-2005}, the gamma kernel was denoted by $\underline{K}^\GAMMA(x,y\mid z,z')$. Here we change the notation to $K^\z(x,y)$. Below we write down the explicit expression $K^\z(x,y)$.

The admissibility condition for $(z,z')$ implies that
$$
\Ga(x+z+\tfrac12)\Ga(x+z'+\tfrac12)>0 \qquad \forall x\in\Z'.
$$
Set
\begin{equation}\label{eq2.B}
A^\z(x):=\frac{\Ga(x+z+\tfrac12)}{\sqrt{\Ga(x+z+\tfrac12)\Ga(x+z'+\tfrac12)}},
\end{equation}
where we choose the positive value of the square root.  Note that
\begin{equation}\label{eq2.Binverse}
(A^\z(x))^{-1}:=\frac{\Ga(x+z'+\tfrac12)}{\sqrt{\Ga(x+z+\tfrac12)\Ga(x+z'+\tfrac12)}}=A^{(z',z)}(x).
\end{equation}

For $z\ne z'$ the kernel $K^\z(x,y)$  is given by the formulas 
\begin{equation}\label{eq2.C1}
K^\z(x,y)=\frac{\sin{\pi z}\sin(\pi z')}{\pi\sin(\pi(z-z'))}\,\dfrac{\dfrac{A^\z(x)}{A^\z(y)}-\dfrac{A^\z(y)}{A^\z(x)}}{x-y}, \quad x\ne y,
\end{equation} 
and
\begin{equation}\label{eq2.C2}
K^\z(x,x)=\frac{\sin(\pi z)\sin(\pi z')}{\pi\sin(\pi(z-z'))}\,(\psi(x+z+\tfrac12)-\psi(x+z'+\tfrac12)), \quad x\in\Z',
\end{equation}
where $\psi(a):=\Ga'(a)/\Ga(a)$ is the logarithmic derivative of the gamma function. Note that \eqref{eq2.C2} is obtained from \eqref{eq2.C1} by dropping the constraint $y\in\Z'$ and taking the limit $y\to x$; here we use the fact that the expression $A^\z(\ccdot)$ is well defined and is analytic in a sufficiently small neighbourhood of an arbitrary integral point. 

These formulas are extended to the case $z=z'=a\in\R\setminus\Z$ by continuity. Namely,
\begin{multline}\label{eq2.C3}
K^{(a,a)}(x,y)=
\left(\dfrac{\sin(\pi a)}{\pi}\right)^2\sgn(\Ga(x+a+\tfrac12)\sgn(\Ga(y+a+\tfrac12))\\
\times\frac{\psi(x+a+\tfrac12)-\psi(y+a+\tfrac12)}{x-y}, \qquad x\ne y,
\end{multline}
and
\begin{equation}\label{eq2.C4}
K^{(a,a)}(x,x)=
\left(\dfrac{\sin(\pi a)}{\pi}\right)^2\, \psi'(x+a+\tfrac12),
\end{equation}
where $\sgn(\ccdot)=\pm1$ is the sign of a nonzero real number.
(In \cite{BO-2005}, the sign factors in the above formula for $K^{(a,a)}(x,y)$ were erroneously missed.)

Evidently, $K^\z(x,y)=K^{(z',z)}(x,y)$. It is also readily seen that the gamma kernel is real and symmetric. 

A nontrivial (and fundamental) property of the gamma kernel is that it is a \emph{projection kernel} (see \cite[Theorem 5.6]{BO-2005} and \cite[\S4]{Ols-2008}). This means that the gamma kernel is the matrix of a selfadjoint projection operator on the Hilbert space $\ell^2(\Z')$. We denote that operator by $\K^\z$. Denoting the natural orthonormal basis of $\ell^2(\Z')$ by $\{e_x: x\in\Z'\}$ one can write
$$
K^\z(x,y)=(\K^\z e_y,e_x)=(\K^\z e_x,e_y), \qquad x,y\in\Z',
$$
where the first equation holds by the very definition, and the second equation follows from the fact that the gamma kernel is real.

\section{The basis vectors $g^\z_m$}\label{sect3}

\subsection{The modified gamma kernel}

If $M$ is a determinantal measure with a correlation kernel $K(x,y)$, then any kernel of the form $\phi(x)K(x,y)\phi^{-1}(y)$, where $\phi(\ccdot)$ is a nonvanishing function, serves as a correlation kernel for $M$ as well. 

We will need the \emph{modified gamma kernel}, which is defined by
\begin{equation}\label{eq3.A}
\wt K^\z(x,y)=\dfrac{A^\z(y)}{A^\z(x)}K^\z(x,y), \quad x,y\in\Z'.
\end{equation}
Note that in the particular case $z=z'=a\in\R\setminus\Z$ one has
$$
A^{(a,a)}(x)=\sgn(\Ga(x+a+\tfrac12)),
$$
and \eqref{eq3.A} turns into
\begin{equation}\label{eq3.A1}
\wt K^{(a,a)}(x,y)=\sgn(\Ga(x+a+\tfrac12))\sgn(\Ga(y+a+\tfrac12))\, K^{(a,a)}(x,y).
\end{equation}

The next proposition provides a series expansion for the modified gamma kernel. Below we set 
\begin{equation}\label{eq3.B1}
C(z,z'):=\frac{(z-z')\sin{\pi z}\sin(\pi z')}{\pi\sin(\pi(z-z'))}, \quad z\ne z',
\end{equation}
and 
\begin{equation}\label{eq3.B2}
C(a,a):=\left(\dfrac{\sin(\pi a)}{\pi}\right)^2, \quad a\in\R\setminus\Z.
\end{equation}

\begin{proposition}\label{prop3.A}
For any $x,y\in\Z'$ one has
\begin{equation}\label{eq3.B}
\wt K^\z(x,y)=C(z,z')\sum_{m=0}^\infty\frac{\Ga(x+z'+\tfrac12+m)}{\Ga(x+z+\tfrac32+m)} \frac{\Ga(y+z+\tfrac12+m)}
{\Ga(y+z'+\tfrac32+m)}, 
\end{equation}
where the series on the right-hand side is absolutely converging. 
\end{proposition}

Note that in the particular case when $z=z'=a$, where $a\in\R\setminus\Z$, the formula is simplified:
\begin{equation}\label{eq3.C1}
\wt K^{(a,a)}(x,y)=C(a,a)\sum_{m=0}^\infty \frac1{(x+a+\tfrac12)(y+a+\tfrac12)}. 
\end{equation}

Proposition \ref{prop3.A} will be derived from the following lemma.

\begin{lemma}\label{lemma3.A}
Let $a,b,c,d$ be complex parameters subject to two constraints:  $a,b\notin\Z_{\le0}$ and $a+b=c+d$. The following summation formula holds 
\begin{equation}\label{eq3.D}
\sum_{m=0}^\infty\frac{\Ga(a+m)\Ga(b+m)}{\Ga(c+m+1)\Ga(d+m+1)}=\frac{\Ga(c)\Ga(d)-\Ga(a)\Ga(b)}{(a-c)(b-c)\Ga(c)\Ga(d)}.
\end{equation}
\end{lemma}

\noindent\textbf{Comments.}  1. The series on the left absolutely converges. To see this,  use the asymptotic formula (\cite[Section 1.18, (4)]{Er})
\begin{equation}\label{eq3.E}
\frac{\Ga(r+\al)}{\Ga(r+\be}\sim r^{\al-\be} \qquad \text{as $r\to+\infty$}.
\end{equation}
From the assumption $a+b=c+d$ it follows that the $m$th term is $O(m^{-2})$. 

2. Formula \eqref{eq3.D} resembles Dougall's summation formula for the bilateral series  ${}_2H_2(1)$ (see \cite[Section 1.4, (1)]{Er} and \eqref{eq3.Dougall} below). Note the crucial r\^ole of the condition $a+b=c+d$; without it, the unilateral series on the left-hand side does not admit a closed expression. 

3. From $a+b=c+d$ it follows that 
\begin{equation}\label{eq3.D1}
(a-c)(b-c)=(a-d)(b-d)=ab-cd. 
\end{equation}
This shows that the right-hand side of \eqref{eq3.D} is symmetric with respect to the switching $c\leftrightarrow d$. 

4. If \eqref{eq3.D1} vanishes, then the  arising indeterminacy $0/0$ on the right-hand side of \eqref{eq3.D} is resolved by continuity.   

5. The left-hand side of \eqref{eq3.D} equals
$$
\frac{\Ga(a)\Ga(b)}{\Ga(c+1)\Ga(d+1)}\,{}_3F_2 (a,b,1; c+1,d+1;1),
$$
so that \eqref{eq3.D} is equivalent to  
$$
{}_3F_2 (a,b,1; c+1,d+1;1)=\frac{cd}{ab-cd}\,\frac{\Ga(c)\Ga(d)-\Ga(a)\Ga(b))}{\Ga(a)\Ga(b)} \qquad \text{for $a+b=c+d$}.
$$
This formula is contained in the handbook \cite{PBM} (formula 7.4.4.28). Unfortunately, \cite{PBM} does not give a reference and we did not manage to find a suitable source. So we give a simple direct proof. 

\begin{proof}[Proof of Lemma \ref{lemma3.A}] 
Suppose first that $(a-c)(b-c)$ does not vanish. Denoting the right-hand side of \eqref{eq3.D} by $F(a,b,c,d)$, it suffices to prove that 
\begin{equation*}
F(a,b,c,d)-F(a+1,b+1,c+1,d+1)=\frac{\Ga(a)\Ga(b)}{\Ga(c+1)\Ga(d+1)}
\end{equation*}
and
\begin{equation*}
\lim_{m\to+\infty}F(a+m,b+m,c+m,d+m)=0.
\end{equation*}
The first relation is verified directly using the basic assumption $a+b=c+d$. The second relation follows from \eqref{eq3.E}. 

To drop the assumption $(a-c)(b-c)\ne0$ we use the fact that the left-hand side of \eqref{eq3.D} is continuous as a function of the parameters.  
\end{proof}

\begin{proof}[Proof of Proposition \ref{prop3.A}]
We apply Lemma \ref{lemma3.A}. 
The series on the right-hand side of \eqref{eq3.B} is of the same form as in \eqref{eq3.D}, with
\begin{equation}\label{eq3.G}
a=x+z'+\tfrac12, \quad b=y+z+\tfrac12, \quad c=x+z+\tfrac12, \quad d=y+z'+\tfrac12.
\end{equation}
It remains to check that the right-hand side of \eqref{eq3.D} matches the definition \eqref{eq3.A} of the modified kernel $\wt K^\z$. 

Suppose first that $x\ne y$ and $z\ne z'$. From \eqref{eq3.A} and \eqref{eq2.C1} it follows that
$$
\wt K^\z(x,y)=\frac{\sin{\pi z}\sin(\pi z')}{\pi\sin(\pi(z-z'))}\,\dfrac{1-\dfrac{(A^\z(y))^2}{(A^\z(x))^2}}{x-y}.
$$
Next, from \eqref{eq2.B} it follows that
$$
(A^\z(x))^2=\frac{\Ga(x+z+\tfrac12)}{\Ga(x+z'+\tfrac12)}.
$$
Therefore, 
\begin{multline*}
\wt K^\z(x,y)=\frac{\sin{\pi z}\sin(\pi z')}{\pi\sin(\pi(z-z'))}\,\dfrac{1-\dfrac{\Ga(x+z'+\tfrac12)\Ga(y+z+\tfrac12)}{\Ga(x+z+\tfrac12)\Ga(y+z'+\tfrac12)}}{x-y}\\
=\frac{\sin{\pi z}\sin(\pi z')}{\pi\sin(\pi(z-z'))}\,\dfrac{1-\dfrac{\Ga(a)\Ga(b)}{\Ga(c)\Ga(d)}}{x-y}=C(z,z')\, \dfrac{1-\dfrac{\Ga(a)\Ga(b)}{\Ga(c)\Ga(d)}}{(z-z')(x-y)},
\end{multline*}
where the second equality follows from \eqref{eq3.G} and the third equality follows from the definition of $C(z,z')$, see \eqref{eq3.B1}.  Finally, from \eqref{eq3.G} it follows that $(z-z')(x-y)=(a-c)(b-c)$. 

This completes the proof in the case when $x\ne y$ and $z\ne z'$. Then these constraints can be removed, because the formulas \eqref{eq2.C2}, \eqref{eq2.C3}, \eqref{eq2.C4} for the gamma kernel are obtained from the basic formula \eqref{eq2.C1} by continuity.
\end{proof}

\subsection{The key inductive relations}
Proposition \ref{prop3.A} allows us to establish relations between the kernels with parameters $\z$ and $(z+m,z'+m)$, which will play a key r\^ole in what follows. 

Introduce the following functions on $\Z'$ depending on admissible parameters $\z$ and indexed by $m\in\Z$:
\begin{equation}\label{eq3.h}
h^\z_m(x):=\frac{\sin(\pi z')\Ga(1+z-z')}{\pi}\,\frac{\Ga(x+z'+m+\tfrac12)}{\Ga(x+z+m+\tfrac32)},
\end{equation}
\begin{multline}\label{eq3.g}
g^\z_m(x):=A^\z(x) h^\z(x)\\
=\frac{\sin(\pi z')\Ga(1+z-z')}{\pi}\, \frac{\Ga(x+z+\tfrac12)}{\sqrt{\Ga(x+z+\tfrac12)\Ga(x+z'+\tfrac12)}} \,\frac{\Ga(x+z'+m+\tfrac12)}{\Ga(x+z+m+\tfrac32)}.
\end{multline}

We are mainly interested in the functions $g^\z_m$, while the functions $h^\z_m$ serve as auxiliary objects.

\begin{lemma}\label{cor3.D-h}
The following two relations hold for the functions $h^\z_m$.
\begin{equation}\label{eq3.r1}
\wt K^\z(x,y)=\sum_{m=0}^\infty h^\z_m(x)h^{(z',z)}_m(y),
\end{equation}
\begin{equation}\label{eq3.r3}
\wt K^\z(x,y)=\sum_{i=0}^{m-1} h^\z_i(x)h^{(z',z)}_i(y)+\wt K^{(z+m,z'+m)}(x,y), \qquad m=1,2,\dots,
\end{equation}
where the series in \eqref{eq3.r1} is absolutely convergent.
\end{lemma}

\begin{corollary}\label{cor3.D}
The following two relations hold for the functions $g^\z_m$.

\begin{equation}\label{eq3.r2}
K^\z(x,y)=\sum_{m=0}^\infty g^\z_m(x)g^{(z',z)}_m(y),
\end{equation}
\begin{multline}\label{eq3.r4}
K^\z(x,y)=\sum_{i=0}^{m-1} g^\z_i(x)g^{(z',z)}_i(y)\\
+ \frac{A^\z(x) A^{(z+m,z'+m)}(y)}{A^{(z+m,z'+m)}(x) A^\z(y)}\, K^{(z+m,z'+m)}(x,y), \qquad m=1,2,\dots,
\end{multline}
where the series in \eqref{eq3.r2} is absolutely convergent.
\end{corollary}

Corollary \ref {cor3.D} will play the main r\^ole in the proof of Theorem \ref{thm3.B} below.

\begin{proof}[Proof of Lemma  \ref{cor3.D-h} and Corollary \ref{cor3.D}.]
Observe that 
\begin{equation}\label{eq3.delta}
\frac{\sin(\pi z')\Ga(1+z-z')}{\pi}\cdot \frac{\sin(\pi z)\Ga(1+z'-z)}{\pi}=C(z,z').
\end{equation}
Indeed,
$$
\Ga(1+z-z')\Ga(1+z'-z)=(z-z')\Ga(z-z')\Ga(1+z'-z)=\frac{\pi (z-z')}{\sin(\pi(z-z'))},
$$
where the last equality follows from Euler's reflection formula \eqref{eq3.L}. Comparing with the definition of $C(z,z')$ (see \eqref{eq3.B1}) we obtain \eqref{eq3.delta}. 

Taking account of \eqref{eq3.delta} and the definition \eqref{eq3.h}, we see that \eqref{eq3.r1} is simply a reformulation of the result of Proposition \ref{prop3.A}.

The relation \eqref{eq3.r2}  follows from \eqref{eq3.r1}. To see this, we multiply both sides of \eqref{eq3.r1} by 
\begin{equation}\label{eq3.factor}
\frac{A^\z(x)}{A^\z(y)}=A^\z(x)A^{(z',z)}(y)
\end{equation}
(the last equality follows from \eqref{eq2.Binverse}) and take account of the link between $\wt K^\z(x,y)$ and $K^\z(x,y)$ (see \eqref{eq3.A}). 

The relation \eqref{eq3.r3} follows from \eqref{eq3.r1} and the fact that
\begin{equation*}
h^\z_{m+k}(x)=(-1)^mh^{(z+m,z'+m)}_k(x), \qquad k, m\in\Z,
\end{equation*}
as it is seen from the definition \eqref{eq3.h}. 

Finally, \eqref{eq3.r4} follows from \eqref{eq3.r3}: we multiply both sides of \eqref{eq3.r3} by \eqref{eq3.factor} and use again \eqref{eq3.A}. 
\end{proof}

\subsection{Orthogonality and biorthogonality}

In the case of principal series, when $z$ and $z'$ are complex conjugated,  the functions $g^\z_m=g^{(z,\bar z)}$ are complex-valued and satisfy the relation
\begin{equation}\label{eq3.bar}
\overline{g^{(z,\bar z)}(x)}=g^{(\bar z, z)}(x).
\end{equation}
In the case of complementary series, when $z$ and $z'$ are real, the functions $g^\z(x)$ are real, too. 

\begin{proposition}\label{prop3.C}
The functions $g^\z_m(x)$, $m\in\Z$,  belong to $\ell^2(\Z')$ and satisfy the biorthogonality relation
\begin{equation}\label{eq3.biorth}
\sum_{x\in\Z'}g^\z_m(x) g^{(z',z)}_n(x)=\begin{cases} 1, & m=n,\\ 0, & m\ne n. \end{cases}
\end{equation}
\end{proposition}

\begin{proof}
Let us check the square summability. 

In the case of the principal series we have
$$
|g^{(z,\bar z)}_m(x)|^2= O(|x|^{-2}), \quad |x|\gg0.
$$
Indeed, for $x\gg0$ this is seen from the definition \eqref{eq3.g} and the asymptotic formula \eqref{eq3.E}. 
The case $x\ll0$ is reduced to the case $x\gg0$ with the aid of Euler's reflection formula (\cite[Section 1.2, (6)]{Er})
\begin{equation}\label{eq3.L}
\Ga(\al)\Ga(1-\al)=\frac{\pi}{\sin(\pi\al)}.
\end{equation}
In the case of the complementary series we obtain in the same way 
$$
(g^\z_m(x))^2=O(|x|^{-2+z'-z}),  \quad |x|\gg0.
$$

These bounds guarantee that $g^\z_m\in\ell^2(\Z')$ in both cases. Note that in the case of the complementary series we have used the fact that $|z'-z|<1$. Note also that that in this case  the functions $g^\z_m(x)$ and $g^{(z',z)}_m(x)$ exhibit different order of decay at infinity.  

To prove the biorthogonality relation \eqref{eq3.biorth} we use Dougall's summation formula for  the two-sided hypergeometric series ${}_2 H_2(1)$ (\cite[Section 1.4, (1)]{Er}):
\begin{equation}\label{eq3.Dougall}
\sum_{k\in\Z}\frac{\Ga(k+a)\Ga(k+b)}{\Ga(k+c)\Ga(k+d)}=\frac{\pi^2}{\sin(\pi a)\sin(\pi b)}\, \frac{\Ga(c+d-a-b-1)}{\Ga(c-a)\Ga(c-b)\Ga(d-a)\Ga(d-b)}.
\end{equation}
This formula holds for any complex parameters $a,b,c,d$ such that $a,b$ are not integral and $\Re(c+d-a-b)>1$. We apply it to
$$
a:=m+z', \quad b:=n+z, \quad c:=m+z+1, \quad d:=n+z'+1
$$
and set $k:=x+\tfrac12$. The required constraints on the parameters are satisfied both for the principal and the complementary series,  and we obtain from \eqref{eq3.Dougall} that the sum on the left-hand side of \eqref{eq3.biorth} equals 
\begin{equation*}
\frac{1}{\Ga(n-m+1)\Ga(m-n+1)}.
\end{equation*}
This expression vanishes for $m\ne n$ and equals $1$ for $m=n$, as desired. 
\end{proof}

\section{The projections $\K^\z$}

\subsection{A link between $\K^\z$ and $\K^{(z+m,z'+m)}$}

Here we explore the relation \eqref{eq3.r4}. Let $V^\z_m$ stand for the operator  of multiplication by the function $A^\z(x)/A^{(z+m,z'+m)}(x)$. From \eqref{eq2.B} it is seen that this function is nonvanishing on $\Z'$ and
\begin{equation}\label{eq3.X}
\dfrac{A^\z(x)}{A^{(z+m,z'+m)}(x)}=1+O(|x|^{-1}), \qquad |x|\gg0.
\end{equation}
This implies that the operator $V^\z_m$ is a bounded operator on $\ell^2(\Z')$ and so is its inverse. Moreover, $V^\z_m-1$ is in the Hilbert-Schmidt class. Note also that $V^\z$ is unitary in the case of the principal series (but not in the case of the complementary series). 

Let $\K^\z_m$ denote the operator with the matrix 
$$
K^\z_m(x,y):=\sum_{i=0}^{m-1} g^\z_i(x)g^{(z',z)}_i(y), \quad x,y\in\Z'.
$$
In other words,
$$
(\K^\z_m f)(x)=\sum_{y\in\Z'}K^\z_m(x,y)f(y), \qquad \forall f\in\ell^2(\Z').
$$
From \eqref{eq3.biorth} it is seen that $\K^\z_m$ has rank $m$ and satisfies the relation$(\K^\z)^2=\K^\z$. Note that in the case of the principal series it is a selfadjoint projection, but in the case of the complementary series it is not selfadjoint, and we use the term  \emph{skew projection}. 

The relation \eqref{eq3.r4} implies the important equality
\begin{equation}\label{kzkmz}
\K^\z=\K^\z_m+V^\z_m \K^{(z+m,z'+m)} (V^\z_m)^{-1}, \qquad m=1,2,\dots\,.
\end{equation}
We are now ready to formulate
\begin{theorem}\label{thm3.B}
Let $\z$ be an arbitrary admissible pair of parameters. For $m=1,2,\dots$, we have 
$$
\Ran(\K^\z)=\Ran(\K^\z_m) \oplus \Ran(V^\z_m \K^{(z+m,z'+m)} (V^\z_m)^{-1})
$$
The space $\Ran(V^\z_m \K^{(z+m,z'+m)} (V^\z_m)^{-1})$ is thus contained in the space $\Ran(\K^\z)$ and has codimension $m$.
\end{theorem}

Recall that $\Ran(\ccdot)$ is our notation for the range of an operator.  In all cases under consideration the range will be always a closed subspace.

\begin{proof}
The equation \ref{kzkmz} represents the projection $\K^\z$ as a sum of two (skew) projections. The following elementary general lemma shows that in this case the sum must be direct.
\end{proof}
\begin{lemma}\label{lemma3.A-bis}
Let $H$ be a vector space and $\{P, P_1, P_2\}$ be a triple of operators on $H$ such that 
$$
P=P_1+P_2, \quad P^2=P, \quad P_1^2=P_1, \quad P_2^2=P_2.
$$
Then\/ $\Ran(P)$ is the direct sum of\/ $\Ran(P_1)$ and\/ $\Ran(P_2)$. 
\end{lemma}

\begin{proof}[Proof of lemma]
From the hypotheses of the lemma it follows that $P_1P_2+P_2P_1=0$. 

On the other hand, the space $H$ can be written as the direct sum of $\Ran(P_1)$ and $\operatorname{Ker}(P_1)$. Writing $P_1$ and $P_2$ in the block form with respect to this decomposition,
$$
P_1=\begin{bmatrix} 1 & 0\\ 0 & 0\end{bmatrix}, \quad P_2=\begin{bmatrix} A & B\\ C & D\end{bmatrix},
$$
we see that the relation $P_1P_2+P_2P_1=0$ precisely means that the blocks $A,B,C$ vanish, and the lemma follows
\end{proof}

\subsection{The case of the principal series: an orthogonal  basis in $\Ran(\K^{(z,\bar z)})$}

In the next theorem we are dealing with the case of the principal series.

\begin{theorem}\label{thm3.A} 
Let $z\in\C\setminus\Z$. The functions $g^{(z,\bar z)}_m$, where $m=0,1,2,\dots$,  constitute an orthonormal basis of\/ $\Ran(\K^{z,\bar z})$. 
\end{theorem}

\begin{proof}
For the principal series, the biorthogonality relation  \eqref{eq3.biorth} precisely means that $\{g^{(z,\bar z)}_m:m\in\Z\}$ is an orthonormal family of functions, as it is seen from  \eqref{eq3.bar}. Next, taking again account of \eqref{eq3.bar}, we may rewrite the relation \eqref{eq3.r2} in the form
$$
K^{(z,\bar z)}(x,y)=\sum_{m=0}^\infty g^{(z,\bar z)}_m(x)\overline{g^{(z,\bar z)}_m(y)}, \quad x,y\in\Z'.
$$
This completes the proof.
\end{proof}

\subsection{The case of the complementary series: the range of $\K^\z$}

The analogue of Theorem \ref{thm3.A} for  the complementary series  is the following 
result.

\begin{theorem}\label{thm3.C}
Suppose that $\ell<z<z'<\ell+1$ for some $\ell\in\Z$. Then the space $\Ran(\K^\z)$ is the closed linear span of the functions $g^\z_0, g^\z_1, g^\z_2,\dots$\,.
\end{theorem}
 The assumption $z<z'$ is made for notational convenience and does not impose a restriction on the gamma kernel because $K^\z(x,y)=K^{(z',z)}(x,y)$. This assumption $z<z'$ will be used on the last step of the proof only. 

\begin{proof}
\emph{Step} 1. Let $H\subset\ell^2(\Z')$ denote the closed linear span of the functions $g^\z_i$, $i=0,1,2,\dots$\,. From the proof of Theorem \ref{thm3.B} it follows that $g^\z_i\in\Ran(\K^\z)$ for any $i=0,1,2,\dots$\,. Thus, $H\subseteq\Ran(\K^\z)$. Now we have to prove that in fact $H=\Ran(\K^\z)$. 

Given $y\in\Z'$, we consider the function $v_y(x):=K^\z(x,y)$. The space $\Ran(\K^\z)$ is the closed linear span of such functions, so that it suffices to prove that $v_y\in H$ for any $y\in\Z'$. 

Fix an arbitrary $y\in\Z'$. Suppose that we have found a sequence $\{v_{y,m}\in H: m=1,2,\dots\}$ such that 
\begin{equation}\label{eq3.V1}
\lim_{m\to\infty}(v_y(x)-v_{y,m}(x))=0 \quad \text{for any fixed $x\in\Z'$}
\end{equation}
and there exists a bound 
\begin{equation}\label{eq3.V2}
\Vert v_{y}-v_{y,m}\Vert \le\const,  \quad \text{uniformly on $m$.}
\end{equation}
Then \eqref{eq3.V1} and \eqref{eq3.V2} would imply that $v_{y,m}\to v_y$ in the weak topology of Hilbert space $\ell^2(\Z')$, which in turn implies that $v_y\in H$, as desired.
\smallskip 

\emph{Step} 2. Our choice of $\{v_{y,m}\}$ is prompted by \eqref{eq3.r4}. Namely, we set 
\begin{equation}\label{eq3.U}
v_{y,m}:=\sum_{i=0}^{m-1}g^\z_i(x)g^{(z',z)}_i(y),  \qquad m=1,2,\dots\,.
\end{equation}
Then we have
\begin{multline}\label{eq3.W}
v_y(x)-v_{y,m}(x)=\frac{A^\z(x) A^{(z+m,z'+m)}(y)}{A^{(z+m,z'+m)}(x) A^\z(y)}\, K^{(z+m,z'+m)}(x,y)\\
=\frac{A^\z(x) A^\z(y+m)}{A^\z(x+m) A^\z(y)}\, K^\z(x+m,y+m).
\end{multline}
Thus, our task is to check \eqref{eq3.V1} and \eqref{eq3.V2} for this concrete expression.
\smallskip

\emph{Step} 3. Let us check \eqref{eq3.V1}. Since
$$
\frac{A^{(z+m,z'+m)}(y)}{A^{(z+m,z'+m)}(x)}\sim \left(\frac{m+y}{m+x}\right)^{(z-z')/2}\to 1,
$$
we have to check that $K^\z(x+m,y+m)\to0$ as $m\to+\infty$. Examine two possible cases: $x\ne y$ and $x=y$. 

In the first case we may use \eqref{eq2.C1}. From it we obtain
$$
|K^\z(x+m,y+m)|\le\const \left|\dfrac{A^\z(m+x)}{A^\z(m+y)}-\dfrac{A^\z(m+y)}{A^\z(m+x)}\right|,
$$
and it suffices to prove that $A^\z(m+x)/A^\z(m+y)$ tends to $1$. Since this quantity is positive for large $m$, we may deal with its square, which is more convenient. We have 
$$
\left(\dfrac{A^\z(m+x)}{A^\z(m+y)}\right)^2=\dfrac{\Ga(m+x+z+\tfrac12)\Ga(m+y+z'+\tfrac12)}{\Ga(m+x+z'+\tfrac12)\Ga(m+y+z+\tfrac12)},
$$
and from the asymptotic formula \eqref{eq3.E} it is seen that this quantity tends to $1$, as desired. 

In the second case we use formula \eqref{eq2.C2}. From it we obtain 
$$
|K^\z(y+m,y+m)|\le\const|\psi(m+y+z+\tfrac12)-\psi(m+y+z'+\tfrac12)|.
$$
Then we apply the asymptotic formula (\cite[Section 1.18, (7)]{Er})
$$
\psi(r)=\log r-\tfrac12 r^{-1}+O(r^{-2}),  \quad r\gg0,
$$ 
which implies that $\psi(m+y+z+\tfrac12)-\psi(m+y+z'+\tfrac12)\to0$, as desired. 
\smallskip

\emph{Step} 4. Let us check \eqref{eq3.V2}. Here we use the assumption $z'>z$. It suffices to show that  
$$
|v_y(x)-v_{y,m}(x)| =O\left(|x|^{-1+\frac{z'-z}2}\right) \qquad \text{for $|x|$ large, uniformly on $m=1,2,\dots$}\,.
$$
Since $(z'-z)/2<1/2$, this will guarantee the uniform convergence of the series 
$$
\sum_{x\in\Z'}(v_y(x)-v_{y,m}(x))^2.
$$

Examine again the expression \eqref{eq3.W}. Discarding the constant factor $A^\z(y)$ in the denominator we may transform \eqref{eq3.W} to the form
$$
A^\z(x)\, \dfrac{1-\left(\dfrac{A^\z(y+m)}{A^\z(x+m)}\right)^2}{x-y}=\dfrac{A^\z(x)}{x-y} - \dfrac{A^\z(x)\,\left(\dfrac{A^\z(y+m)}{A^\z(x+m)}\right)^2}{x-y}.
$$
We have
$$
\left|\dfrac{A^\z(x)}{x-y} \right|\sim|x|^{-1+(z-z')/2}=O(|x|^{-1}),
$$
since $z<z'$. Next,
$$
\left(\dfrac{A^\z(y+m)}{A^\z(x+m)}\right)^2 \le \const m^{z-z'}|x+m|^{z'-z}.
$$
Therefore,
\begin{multline*}
\dfrac{\left|  A^\z(x)\,\left(\dfrac{A^\z(y+m)}{A^\z(x+m)}\right)^2\right|}{|x-y|}\le \const |x|^{-1+(z-z')/2}m^{z-z'}|x+m|^{z'-z}\\
=\const\, \left(\dfrac{|x+m|}{|x|m}\right)^{z'-z}\, |x|^{-1+\frac{z'-z}2}.
\end{multline*}
For large $|x|$ and $m$ we have 
$$
\dfrac{|x+m|}{|x|m}<1,
$$
and since $z'-z>0$ we finally obtain the desired bound. 

This completes the proof. 
\end{proof}

\section{Remarks on shift invariant subspaces and Schur measures}

\subsection{Shift invariant subspaces of $\ell^2(\Z)$}\label{sect4.1}
Let $\T$ be the unit circle $|u|=1$ in $\C$ and $L^2(\T)$ stand for the complex $L^2$ space corresponding to the normalized Lebesgue measure on $\T$. The Laurent monomials $u^m$, $m\in\Z$, form an orthonormal basis in $L^2(\T)$. The \emph{Hardy space} $H^2(\T)\subset L^2(\T)$ is the closed subspace spanned by the monomials $u^m$ with exponent $m\in\Z_{\ge0}$. 

We denote by $f\mapsto \wh f$ the Fourier transform $L^2(\T)\to\ell^2(\Z)$ sending $u^m$ to the delta function at $m\in\Z$. The \emph{shift operator} on $\ell^2(\Z)$ is the unitary operator $S$ defined by $S g(k)=g(k-1)$. Its inverse Fourier transform $\check S$ is the unitary operator on $L^2(\T)$ given by $\check S f(u)=uf(u)$. 

We say that a closed subspace $L\subseteq \ell^2(\Z)$ is \emph{shift invariant} if $SL\subseteq L$. The shift invariant subspaces split into two classes depending on whether $S L=L$ or $SL$ is strictly contained in $L$. Adopting the terminology of \cite[Lecture II]{Helson} we will speak about \emph{doubly invariant} and \emph{simply invariant} subspaces, respectively. Similar subspaces arise in $L^2(\T)$, with $S$ replaced by $\check S$. 

\begin{proposition}\label{prop4.A}
{\rm(i)} The doubly invariant subspaces $\mathcal L\subseteq L^2(\T)$ are determined by Lebesgue measurable sets $A\subseteq\T$; the subspace $\mathcal L_A$ corresponding to a given set $A$ is formed by functions vanishing outside $A$.

{\rm(ii)} The simply invariant subspaces $\mathcal L\subseteq L^2(\T)$ are determined by complex valued measurable functions $\varphi(u)$ on $\T$ such that $|\varphi(u)|=1$ almost everywhere; the subspace corresponding to a given function $\varphi$ has the form $\mathcal L_\varphi:=\varphi\cdot H^2(\T)$.
\end{proposition}

These are well-known classical results. Claim (i) is due to Wiener (\cite[Lecture II, Theorem 2]{Helson}), and claim (ii) is a generalization of Beurling's theorem due to Helson and Laudenslager (\cite[Lecture II, Theorem 3]{Helson}). 

As a corollary we obtain the description of projection kernels on $\Z$ corresponding to shift invariant subspaces of both types: 

$\bullet$ The kernels corresponding to doubly invariant subspaces $L\subseteq\ell^2(\Z)$ have the form
\begin{equation}\label{eq4.A}
\mathcal K_A(a,b)=\wh{\one_A}(a-b), \qquad a,b\in\Z,
\end{equation}
where $A\subseteq\T$ is a measurable set and $\one_A$ is its characteristic function.

$\bullet$ The kernels corresponding to simply invariant subspaces $L\subseteq\ell^2(\Z)$ have the form
\begin{equation}\label{eq4.B}
\mathcal K_\varphi(a,b)=\sum_{n=0}^\infty\wh\varphi(a-n)\overline{\wh\varphi(b-n)}, \qquad a,b\in\Z,
\end{equation}
where $\varphi(u)$ is a complex valued measurable function on $\T$ such that $|\varphi(u)|=1$ almost everywhere. 

A concrete example of type \eqref{eq4.A} is the \emph{discrete sine kernel} 
$$
\mathcal K(a,b)=\frac{\sin(\al(a-b))}{\pi(a-b)}, \qquad 0<\al<\pi,
$$
which corresponds to the arc $A=\exp\{[-i\al, i\al]\}\subset\T$. About this kernel see \cite{BOO}. 

Under a suitable identification of $\Z'$ with $\Z$, the modified gamma kernel discussed in Section \ref{sect3} provides an example of type \eqref{eq4.B}, see the next subsection.  

The kernels of the form \eqref{eq4.A} are precisely translation invariant projection kernels. The kernels of the form \eqref{eq4.B}, on the contrary, are not translation invariant.

\begin{question}
What can be said about determinantal measures with general  projection correlation kernels of type \eqref{eq4.A} or \eqref{eq4.B}? 
\end{question}

For translation invariant kernels, there is an interesting result \cite[Corollary 7.14]{Ly-2003}, see also \cite[Section 4.1]{Ly-2014}.

\subsection{Schur measures}
Let $\Sym$ denote the algebra of symmetric functions \cite{Mac-1995}. It is generated by the power sum functions $p_1,p_2,\dots$  and possesses a distinguished basis formed by the Schur symmetric functions $s_\la$ indexed by partitions $\la\in\Y$. Another set of generators of $\Sym$ is formed by the complete homogeneous symmetric functions $h_1,h_2,\dots$, with the formal generating series  
$$
H(u):=\sum_{n=0}^\infty h_nu^n=\exp\left(\sum_{k=1}^\infty \frac1k p_k u^k\right).
$$

An algebra homomorphism $\epsi:\Sym\to\C$ is called a \emph{specialization}. It is uniquely determined by the choice of complex numbers $\epsi(p_k)$, $k=1,2,\dots$\,. Alternatively, one can specify the numbers $\epsi(h_k)$. The connection between the two sequences $\{\epsi(p_k)\}$ and $\{\epsi(h_k)\}$ is given by 
\begin{equation}\label{eq4.F-bis}
\epsi(H(u)):=1+\sum_{k=1}^\infty \epsi(h_k) u^k=\exp\left(\sum_{k=1}^\infty \frac1k \epsi(p_k) u^k\right).
\end{equation}
The Jacobi--Trudi identity \cite[chapter I, (3.4)]{Mac-1995} implies 
$$
\epsi(s_\la)=\det[\epsi(h_{\la_i-i+j})], \qquad \la\in\Y,
$$
with the understanding that $h_0:=1$ and $h_{-1}=h_{-2}=\dots:=0$; the order of the determinant is any number greater or equal to the number of nonzero parts of partition $\la$.

\begin{definition}[See \cite{Ok-2001a}, \cite{BOk-2000}]\label{def4.A}
Let $\epsi$ and $\epsi'$ be two specializations subject to two conditions: 
\begin{equation}\label{eq4.C}
 \epsi(s_\la)\epsi'(s_\la)\ge0, \quad \forall \la\in\Y; \qquad 
 \sum_{\la\in\Y}\epsi(s_\la)\epsi'(s_\la)<\infty.
\end{equation} 
The corresponding  \emph{Schur measure} is the probability measure on $\Y$ is given by 
the formula
$$
M(\la):=\frac1Z \epsi(s_\la)\epsi'(s_\la), \qquad \la\in\Y,
$$
where
$$
Z:=\sum_{\la\in\Y}\epsi(s_\la)\epsi'(s_\la)=\exp\left(\sum_{k=1}^\infty \frac1k \epsi(p_k)\epsi'(p_k)\right).
$$
\end{definition}

Note that the first condition in \eqref{eq4.C} is satisfied if $\epsi'=\bar\epsi$, meaning that $\epsi'(p_k=\overline{\epsi(p_k)}$ for all $k$. 

\begin{proposition}\label{prop4.B}
{\rm(i)} For any Schur measure $M\in\P(\Y)$, its pushforward under the embedding $\Y\to\Om$ defined by \eqref{eq1.B}  is a determinantal measure. 

{\rm(ii)} Let $\epsi$ and $\epsi'$ be the specializations corresponding to $M$. Suppose that the series $\epsi(H(u))$ and $\epsi'(H(u))$  defined by \eqref{eq4.F} converge in a neighbourhood of the unit circle $\T\subset\C$. Introduce two functions on $\T$ defined by $(\epsi,\epsi')$:
\begin{equation}\label{eq4.D}
\Phi(u):=\frac{\epsi(H(u))}{\epsi'(H(u^{-1}))}=\frac{\sum_{n=0}^\infty \epsi(h_n)u^n}{\sum_{n=0}^\infty \epsi'(h_n)u^{-n}}
\end{equation}
and
\begin{equation}\label{eq4.E}
\Phi'(u):=\frac{\epsi'(H(u))}{\epsi(H(u^{-1}))}=\frac{\sum_{n=0}^\infty \epsi'(h_n)u^n}{\sum_{n=0}^\infty \epsi(h_n)u^{-n}}=(\Phi(u^{-1}))^{-1}.
\end{equation}
Finally, let $\wh\Phi(k)$ and\/ $\wh{\Phi'}(k)$ be their Fourier coefficients. 

The kernel 
\begin{equation}\label{eq4.F}
K(x,y)=\sum_{n=0}^\infty \wh\Phi(x+n+\tfrac12)\wh{\Phi'}(y+n+\tfrac12), \qquad x,y\in\Z',
\end{equation}
serves as a correlation kernel for $M$. 
\end{proposition}

\begin{proof} 
These results are due to Okounkov \cite{Ok-2001a}; see also \cite{BOk-2000} for more details. The expression \eqref{eq4.F} is obtained from the generation series displayed after formula (3.3) in \cite{BOk-2000}. For another approach to Schur measures, see \cite[Section 3]{Jo}.
\end{proof} 

\begin{corollary}
Suppose additionally that $\epsi'=\bar \epsi$. Then $|\Phi(u)|\equiv1$ on $\T$ and \eqref{eq4.F} takes the form
\begin{equation}\label{eq4.G}
K(x,y)=\sum_{n=0}^\infty \wh\Phi(x+n+\tfrac12)\overline{\wh{\Phi}(y+n+\tfrac12)}, \qquad x,y\in\Z'.
\end{equation}
This is a projection kernel corresponding to a subspace of $\ell^2(\Z')$. The functions $\wh\Phi_n(x):=\wh\Phi(x+n+\tfrac12)$ with $n=0,1,2,\dots$ form an orthonormal basis of that subspace. 
\end{corollary}

Comparing \eqref{eq4.G} with \eqref{eq4.B} we see that these two expressions are the same, up to a minor adjustment: one has to make a change of variables $a=-(x+\tfrac12)$ and set $\varphi(u)=\Phi(u^{-1})$.  Thus, the Schur measures with $\epsi'=\bar\epsi$ fit into the formalism described in the previous subsection.

\begin{example}
The z-measure $M^\z_\xi$ (Definition \ref{def2.B}) is a Schur measure corresponding to  the specializations
$$
\epsi(p_k)=z\xi^{k/2}, \quad \epsi'(p_k)=z'\xi^{k/2}, \qquad k=1,2,\dots,
$$
and the corresponding function \eqref{eq4.D} is
\begin{equation}\label{eq4.H}
\Phi(u)=\Phi^\z_\xi(u):=\frac{(1-u^{-1}\xi^{\frac12})^{z'}}{(1-u\xi^{\frac12})^z}.
\end{equation}
\end{example}

Our final remark concerns the limit transition as $\xi\to1$. As pointed out in Subsection \ref{sect2.2}, in this limit the z-measures $M^\z_\xi$ converge to the gamma kernel measure $M^\z$. Therefore, it is natural to examine the limit of $\Phi^\z_\xi$. It is given by 
$$
\Phi^\z(u):=\frac{(1-u^{-1})^{z'}}{(1-u)^z}.
$$
This function is discontinuous at $u=1$ and hence cannot be extended to a holomorphic function in an annulus around of $\T$. This is related to the fact that the gamma kernel measure is not a Schur measure.

However, the function $\Phi^\z(u)$ is still integrable on $\T$, and one can prove that its Fourier coefficients are given by
\begin{equation*}
\wh{\Phi^\z}(k)=\frac{\sin(\pi z')\Ga(1+z-z')}{\pi}\,\frac{\Ga(k+z')}{\Ga(k+z+1)}.
\end{equation*}
Likewise,
\begin{equation*}
\wh{\Phi^{(z', z)}}(k)=\frac{\sin(\pi z)\Ga(1+z'-z)}{\pi}\,\frac{\Ga(k+z)}{\Ga(k+z'+1)}.
\end{equation*}
Substituting these two expressions into \eqref{eq4.F} instead of $\wh\Phi(\ccdot)$ and $\wh{\Phi'}(\ccdot)$ we obtain precisely the series expansion \eqref{eq3.B} of the modified gamma kernel. 

It follows that in the case of the principal series, when $z'=\bar z$, the modified gamma kernel fits into the formalism of Subsection \ref{sect4.1}.

\section{Multiplicative functionals}\label{sect5}

In this section $\X$ is a countable set with no additional structure. As usual, we set $\Om:=\{0,1\}^\X$ and denote by $\ell^2(\X)$ the coordinate Hilbert space with the distinguished orthonormal basis $\{e_x\}$ labeled by $\X$. Given a closed  subspace $L$ of $\ell^2(\X)$, we denote by $K^L$ the operator of orthogonal projection onto $L$. Let $K^L(x,y)=(K^L e_y,e_x)$ be the matrix of $K^L$. The determinantal measure on $\Om$ with the correlation kernel $K^L(x,y)$ will be denoted by $M[L]$.  

Suppose $\al(x)$ is a complex-valued function on $\X$ such that its modulus $|\al(\ccdot)|$ is bounded away from $0$ and $\infty$. Then the subspace $\al L:=\{\al f: f\in L\}$ is closed in $\ell^2(\X)$. Our aim in this section is to show that, under suitable assumptions on the function $\al$, the measure $M[\al L]$ is absolutely continuous with respect to $M[L]$. Moreover, the corresponding Radon--Nikod\'ym derivative can be described explicitly --- it is given by what is called a normalized multiplicative functional. For the precise formulation see Theorem \ref{thm5.A} below. This theorem is a particular case of a more general result contained in \cite{Bu} (see there Proposition 4.2 and related material). However, we decided to present a proof, because in the case of discrete space $\X$ (which is of interest to us), the arguments of \cite{Bu} can be greatly simplified. 

\subsection{A transformation of projection operators}

Let $H$ be a complex Hilbert space, $L$ be a closed subspace of $H$ and $K$ be the operator of orthogonal projection onto $L$. Next, let $A$ be a bounded invertible operator on $H$. Then the subspace $\wt L:=AL$ is closed. Let $\wt K$ denote the operator of orthogonal projection onto $\wt L$. 

\begin{lemma}\label{lemma5.A}
The following formulas hold
\begin{equation}\label{eq5.A}
\wt K=AK(1+(A^*A-1)K)^{-1}A^*=A(1+K(A^*A-1))^{-1}KA^*.
\end{equation}
\end{lemma}

\begin{proof}
Set
$$
\wt K_1:=AK(1+(A^*A-1)K)^{-1}A^*, \quad \wt K_2:=A(1+K(A^*A-1))^{-1}KA^*.
$$

First of all, observe that the operators $1+(A^*A-1)K$ and $1+K(A^*A-1)$ are invertible, so that $\wt K_1$ and $\wt K_2$ are well defined. Indeed, write $A^*A$  as a $2\times2$ operator matrix corresponding to the orthogonal decomposition $H=L\oplus L^\perp$:
$$
A^*A=\begin{bmatrix} A_{11} & A_{12}\\ A_{21} & A_{22} \end{bmatrix}.
$$
Then
$$
1+(A^*A-1)K=\begin{bmatrix} A_{11} & 0\\ A_{21} & 1 \end{bmatrix}.
$$
Since $A^*A\ge\epsi$ for $\epsi>0$ small enough, we also have $A_{11}\ge\epsi$ and hence $1+(A^*A-1)K$ is invertible. 
Likewise, $1+K(A^*A-1)$ is invertible, too.

Next, in the same notation we have
$$
K(1+(A^*A-1)K)^{-1}=\begin{bmatrix} A^{-1}_{11} & 0\\ 0 & 1\end{bmatrix}=(1+K(A^*A-1))^{-1}K,
$$
whence $\wt K_1=\wt K_2$, which is the second equality in \eqref{eq5.A}. 

From the very definition of $\wt K_1$ it is readily checked that $\wt K_1 AK=AK$. This shows that $\wt K_1$ acts identically on $\wt L=AL$.

Finally, if a vector $\xi\in H$ belongs to the orthogonal complement $\wt L^\perp$, then $A^*\xi\in L^\perp$ and $KA^*\xi=0$. It follows that $\wt K_2\xi=0$. 

We have proved that the operator $\wt K_1=\wt K_2$ fixes all vectors of $\wt H$ and equals zero on $\wt L^\perp$. Therefore, it coincides with $\wt K$. 
\end{proof}

\subsection{A characteristic property of determinantal measures}

Let $a(x)$ be a function on $\X$ such that $a(x)=1$ outside a finite set. We assign to it the following function on $\Om$:
\begin{equation}\label{eq5.F}
\Psi_a(\om):=\prod_{x\in\om} a(x), \quad \om\in\Om.
\end{equation}
Note that the product is in fact finite. It defines a cylinder function on $\Om$ which we call a \emph{multiplicative functional}. For a measure $M\in\P(\Om)$, we will denote by the symbol $\E_M(\ccdot)$ the corresponding expectation. The following fact is well known.

\begin{lemma}\label{lemma5.B}
Let $M\in\P(\Om)$, $K$ be a bounded operator on $\ell^2(\X)$, and $K(x,y)$ be its matrix. 

{\rm(i)} If $M$ is a determinantal measure and $K(x,y)$ serves as a correlation kernel for $M$, then
\begin{equation}\label{eq5.B}
\E_M(\Psi_a)=\det[1+(a-1)K]
\end{equation}
for any function $a$ on $\X$ such that $a-1$ is finitely supported. 

{\rm(ii)} Conversely, if \eqref{eq5.B} holds for any function $a(\ccdot)$ as above, then $M$ is a determinantal and  $K(x,y)$ serves as its correlation kernel. 
\end{lemma}

\begin{proof}
See, e.g., \cite[Lemma 2.6]{Ols-2011}. 
\end{proof}

\subsection{Normalized multiplicative functionals}

Suppose that $a(\ccdot)-1$ is finitely supported (as in Lemma \ref{lemma5.B}) and, moreover, $a(x)>0$ for all $x\in\X$. Given a measure $M\in\P(\Om)$, we define the \emph{normalized multiplicative functional} $\PPsi_a$ by
\begin{equation}\label{eq5.G}
\PPsi_{a,M}(\om):=\frac{\Psi_\al(\om)}{\E_M(\Psi_\al)}, \qquad \om\in\Om.
\end{equation}
It is well defined, because $\Psi_a$ is continuous and strictly positive, which implies that $\E_M(\Psi_a)$ is finite and strictly positive.  

\begin{lemma}\label{lemma5.C}
Let $L\subset\ell^2(\X)$ be a closed subspace and $a(x)$ be a strictly positive function on $\X$ such that $a(\ccdot)-1$ is finitely supported. We have
\begin{equation}\label{eq5.C}
M[\sqrt a L]=\PPsi_{a,M} \cdot M[L].
\end{equation}
\end{lemma}

\begin{proof}
Let $K$ and $\wt K$ be the orthogonal projections onto $L$ and $\sqrt a L$, respectively. By Lemma \ref{lemma5.A}, applied to the operator of multiplication by $a(\ccdot)$ (which we denote still by $a$), 
$$
\wt K=\sqrt a K(1+(a-1)K)^{-1}\sqrt a.
$$

We abbreviate $M:=M[L]$ and $\wt M:=\PPsi_{a,M} \cdot M[L]$. Note that $\wt M\in\P(\Om)$. By virtue of Lemma \ref{lemma5.B}, it suffices to prove that
$$
\E_{\wt M}(\Psi_b)=\det[1+(b-1)\wt K],
$$
where $b$ is an arbitrary function on $\X$ such that $b-1$ is finitely supported. The idea is to compute $\E_M(\Psi_{ba})$ in two ways. 

First, we have
$$
\E_M(\Psi_{ba})=\det[1+(ba-1)K].
$$

Second, since $\Psi_{ba}=\Psi_b\Psi_a$, we obtain
$$
\E_M(\Psi_{ba})=\det[1+(a-1)K] \E_{\wt M}(\Psi_b).
$$
Comparing the two equalities we obtain
$$
\E_{\wt M}(\Psi_b)=\frac{\det[1+(ba-1)K]}{\det[1+(a-1)K]}.
$$

It remains to check the relation
\begin{equation}\label{eq5.D}
\det[1+(b-1)\wt K]=\frac{\det[1+(ba-1)K]}{\det[1+(a-1)K]}.
\end{equation}
To do this we transform 
$$
1+(ba-1)K=1+(a-1)K+(b-1)aK=\big(1+(b-1)aK(1+(a-1)K)^{-1}\big)(1+(a-1)K).
$$
It follows that the right-hand side of \eqref{eq5.D} equals
$$
\det[1+(b-1)aK(1+(a-1)K)^{-1}].
$$
Since $(b-1)a=a(b-1)$, we may rewrite this as
$$
\det[1+(b-1)\sqrt a K(1+(a-1)K)^{-1}\sqrt a]=\det[1+(b-1)\wt K].
$$
This proves \eqref{eq5.D} and completes the proof. 
\end{proof}

\subsection{Absolute continuity  of determinantal measures}

Let, as above, $M:=M[L]$, where $L$ is a closed subspace of $\ell^2(\X)$, $K$ be the operator of orthogonal projection onto $L$, and $K(x,y)$ be the matrix of $K$. 

The next theorem extends Lemma \ref{lemma5.C} to a wider class of functions $a(x)$. Note that this theorem is a particular case of  Proposition 4.2 in  \cite{Bu}, which, in turn, is a generalization of Proposition 2.1 in Bufetov \cite{buf-era-ms}, previously announced in \cite{buf-umn}. 
 
In the discrete setting the proof is simpler than the general argument contained in \cite{Bu}, \cite{buf-era-ms}.  
For this reason and for completeness of the exposition, we give the proof here.

\begin{theorem}\label{thm5.A}
Let $M=M[L]$, where $L$ is a closed subspace of $\ell^2(\X)$, and let $\al(x)$ be a non-vanishing complex-valued function on $\X$ such that  $|\al(\ccdot)|^2-1$ belongs to $\ell^2(\X)$. 

The measure $M[\al L]$ is absolutely continuous with respect to $M$ and one has
\begin{equation}\label{eq5.H}
M[\al L]=\PPsi_{|\al|^2,M}\cdot M.
\end{equation}
In particular, $M[\al L]$ is absolutely continuous with respect to $M$.
\end{theorem}

We will deduce the theorem from the following proposition, which is of independent interest. 

\begin{proposition}\label{prop5.A}
Fix a sequence $\X_1\subset\X_2\subset\dots$ of finite subsets of\/ $\X$ such that their union is the whole set $\X$. Given a function $a(x)$ on $\X$, we define its $n$th truncation $a_n(x)$, where $n=1,2,\dots$, by
$$
a_n(x):=\begin{cases} a(x), & x\in\X_n,\\
 1, & x\in\X\setminus\X_n. \end{cases}
$$
Next, suppose that $a(x)$ is strictly positive on $\X$ and $a(\ccdot)-1$ belongs to $\ell^2(\X)$. 

Then there exists a limit
\begin{equation}\label{eq5.E}
\PPsi_{a,M}:=\lim_{n\to\infty} \PPsi_{a_n, M}
\end{equation}
in the norm topology of the Banach space $L^1(\Om,M)$, and this limit does not depend on the choice of the sequence  $\{\X_n\}$.
\end{proposition}

\begin{remark}\label{rem5.A}
Note that the product \eqref{eq5.F}, which defines the functional $\Psi_a(\om)$, converges for all $\om\in\Om$ only if  $a(x)$ satisfies the stronger condition $a(\ccdot)-1\in\ell^1(\X)$. Proposition \ref{prop5.A} provides a way to regularize a possibly divergent infinite product, for $M$-almost all $\om$'s and up to an overall constant factor. In a number of concrete situations (including that of Theorem \ref{thm6.A} below), one just needs to deal with functions $a(x)$ for which $a(\ccdot)-1$ is in $\ell^2(\X)$ but not in $\ell^1(\X)$. 
\end{remark}

\begin{proof}[Derivation of Theorem \ref{thm5.A} from Proposition \ref{prop5.A}]
Suppose first that $\al(x)$ is real-valued and positive. Set $a(x):=\al^2(x)$ and write $\sqrt a$ instead of $\al$. By Lemma \ref{lemma5.C},
$$
M[\sqrt{a_n} L]= \PPsi_{a_n,M}\cdot M.
$$
Denote by $\wt K$ and $\wt K_n$ the operators of orthogonal projections on the subspaces $\sqrt a L$ and  $\sqrt{a_n}L$, respectively. The condition on $a(\ccdot)$ implies that $a(x)\to1$ at infinity, which in turn implies that the functions $a_n(x)$ approximate the function $a(x)$ in the supremum norm. Therefore, $\wt K_n\to\wt K$ in the norm topology, which in turn implies that the measures $M[\sqrt{a_n} L]$ weakly converge to the measure $M[\sqrt a L]$. On the other hand, \eqref{eq5.E} implies that the measures $\PPsi_{a_n,M}\cdot M$ weakly converge to the measure $\PPsi_{a,M}\cdot M$. This proves \eqref{eq5.H} in the case of positive function $\al(x)$. The general case is reduced to that case, because if $\phi(x)$ is a complex-valued function with $|\phi(x)|\equiv1$, then $M[\al\phi L]=M[\al L]$. Indeed, the corresponding orthogonal projections differ from each other by conjugation by the operator of multiplication by $\phi$, which does not affect the correlation functions. 
\end{proof}

\begin{proof}[Proof of Proposition \ref{prop5.A}]
If the existence of the limit \eqref{eq5.E} is already established, then it is easy to see that it does not depend on the choice of the sequence of nested finite subsets. Indeed, given two different sequences,  $\{\X'_n\}$ and $\{\X''_n\}$, it suffices to build a third sequence $\{\X_n\}$ by taking alternately subsets from $\{\X'_n\}$ and $\{\X''_n\}$ with sufficiently large numbers.
 
We proceed to the proof of \eqref{eq5.E}. Let $\Vert\ccdot\Vert_1$ stand for the norm of the Banach space $L^1(\Om,M)$. It suffices to show that
\begin{equation}\label{eq5.J}
\lim_{m,n\to\infty}\Vert \PPsi_{a_m,M}-\PPsi_{a_n,M}\Vert_1=0.
\end{equation}

\emph{Step} 1. In this step we reduce \eqref{eq5.J} to two claims concerning a different version of multiplicative functionals. 

Let $b(x)$ be a positive function on $\X$ such that $b(\ccdot)-1$ is finitely supported. We set 
\begin{equation}\label{eq5.I}
\wt\Psi_{b,M}(\om):=\frac{\Psi_b(\om)}{e^{\tr((\log b)K)}}, \quad \om\in\Om,
\end{equation}
where $K$ is the operator of orthogonal projection onto $L$. 

For our purpose,  the functionals $\wt\Psi_{b,M}$ are more convenient to deal with than the functionals $\PPsi_{b,M}$. This is due to the fact that $\wt\Psi_{b,M}$, like $\Psi_b$,  is multiplicative with respect to $b$. Namely, if $b'$ is another function of the same type, then we obviously have
\begin{equation}\label{eq5.I1}
\wt\Psi_{bb',M}=\wt\Psi_{b,M}\wt\Psi_{b',M}. 
\end{equation}

We are going to prove that 
\begin{equation}\label{eq5.K}
\lim_{m,n\to\infty}\Vert \wt\Psi_{a_m,M}-\wt\Psi_{a_n,M}\Vert_1=0
\end{equation}
and 
\begin{equation}\label{eq5.L}
\text{there exists a limit $\lim_{n\to\infty}\E_M(\wt\Psi_{a_n,M})>0$}.
\end{equation}
Because
\begin{equation*}
\PPsi_{b,M}=\dfrac{\wt\Psi_{b,M}}{\E_M(\wt\Psi_{b,M})},
\end{equation*}
the desired claim \eqref{eq5.J} will follow from \eqref{eq5.K} and \eqref{eq5.L}. 
\smallskip

\emph{Step} 2. Let again $b(x)$ be an arbitrary positive function on $\X$ such that $b(\ccdot)-1$ is finitely supported. In this step we will write $\E(\wt\Psi_{b,M})$ in a form convenient for further applications. Recall the definition of the \emph{Hilbert--Carleman regularized determinant} $\detreg(\ccdot)$, see \cite{GGK} or \cite{GK}. For a trace class operator $A$ on a Hilbert space, the definition is
$$
\detreg(1+A):=\det(1+A)\exp(-\tr A).
$$
It is well known (\cite{GGK}, \cite{GK}) that this expression is continuous in the Hilbert--Schmidt metric and extends by continuity to the space of Hilbert--Schmidt operators. Further, if $A$ is selfadjoint, then
\begin{equation}\label{eq5.N}
\detreg(1+A)=\prod(1+\la_i)e^{-\la_i},
\end{equation}
where $\{\la_i\}$ are the eigenvalues of $A$ counted with their multiplicities (the product converges because $A$  is Hilbert--Schmidt). We will use these facts shortly.

With this definition in hand we may write
\begin{equation}\label{eq5.M}
\E_M(\wt\Psi_{b,M})=\detreg(1+(b-1)K)\cdot \exp(\tr((b-1-\log b)K)).
\end{equation}
Indeed,  we have
\begin{equation*}
\E_M(\wt\Psi_{b,M})=\dfrac{\det(1+(b-1)K)}{e^{\tr((\log b)K)}}=
\dfrac{\det(1+(b-1)K)}{\exp(\tr((b-1)K))}\cdot \exp(\tr((b-1-\log b)K)),
\end{equation*}
where the first equality follows from \eqref{eq5.I} and \eqref{eq5.B}, and the second equality is a trivial transformation; the final result is precisely the right-hand side of \eqref{eq5.M}. 
\smallskip

\emph{Step} 3. Now we can prove \eqref{eq5.L}. More precisely, we will show that   
\begin{equation}\label{eq5.L1}
\lim_{n\to\infty}\E_M(\wt\Psi_{a_n,M}) = \detreg(1+(a-1)K)\cdot \exp(\tr((a-1-\log a)K))
\end{equation}
and that the right-hand side is strictly positive. 

Indeed, the condition $a(\ccdot)-1\in\ell^2(\X)$ precisely means that the operator of multiplication by the function $a(x)-1$ is Hilbert--Schmidt. Therefore, the operator $(a-1)K$ is Hilbert--Schmidt, too. It follows that $\detreg(1+(a-1)K)$ is well defined. 

By the very definition of the functions $a_n(x)$, the functions $a_n(x)-1$ approximate $a(x)-1$ in the $\ell^2$ metric. This implies that $(a_n-1)K\to (a-1)K$ in the Hilbert--Schmidt metric. Therefore, 
$$
\detreg(1+(a_n-1)K)\to \detreg(1+(a-1)K).
$$

Since the function $a(x)$ is strictly positive and tends to $1$ at infinity, it is bounded away from $0$. It follows that the spectrum of the selfadjoint operator $K(a-1)K$ is bounded from below from $-1$, so that 
$$
\det(1+K(a-1)K)>0
$$
(here we also used \eqref{eq5.N}). Therefore,  
$$
\detreg(1+(a-1)K)=\detreg(1+(a-1)K^2)=\detreg(1+K(a-1)K)>0
$$
(the latter equality is one more property of the regularized determinant). 

It remains to check that  $\tr((a-1-\log a)K)$ is well defined and 
$$
\tr((a_n-1-\log a)K)\to \tr((a-1-\log a)K).
$$
Both claims follow from the fact that the function  $a(x)-1-\log a(x)$ belongs to $\ell^1(\X)$ and
$$
a_n(\ccdot)-1-\log a(\ccdot)\; \longrightarrow\;  a(\ccdot)-1-\log a(\ccdot)
$$
in the metric of $\ell^1(\X)$. 
\smallskip

\emph{Step} 4. In this last step we prove \eqref{eq5.K}. Let $\Vert\ccdot\Vert_2$ and $(\ccdot,\ccdot)$ denote the norm and scalar product in $L^2(\Om,M)$.  Using the multiplicativity property \eqref{eq5.I1} we have 
\begin{multline*}
\Vert \wt\Psi_{a_m,M}-\wt\Psi_{a_n,M}\Vert_1=\Vert \wt\Psi_{a_n,M}(\wt\Psi_{a_m/a_n,M}-1)\Vert_1
=(\wt\Psi_{a_n,M},\, \vert\wt\Psi_{a_m/a_n,M}-1|)\\
\le \Vert \wt\Psi_{a_n,M}\Vert_2\cdot\Vert(\wt\Psi_{a_m/a_n,M}-1)\Vert_2.
\end{multline*}

Next,
$$
\Vert \wt\Psi_{a_n,M}\Vert_2^2=\E_M(\wt\Psi_{a^2_n,M}),
$$
and this quantity is uniformly bounded by virtue of \eqref{eq5.L1}. 

Finally,
$$
\Vert(\wt\Psi_{a_m/a_n,M}-1)\Vert^2_2=\E_M(\wt\Psi_{(a_m/a_n)^2,M})+1-2\E_M(\wt\Psi_{a_m/a_n,M}).
$$
We have to prove that this quantity goes to $0$ as $m,n\to\infty$, and for this it suffices to prove that 
$$
\E_M(\wt\Psi_{(a_m/a_n)^2,M})\to1, \quad \E_M(\wt\Psi_{a_m/a_n,M})\to1.
$$
To see this we use the fact that the differences $(a_m/a_n)^2-1$ and $a_m/a_n-1$ tend to $0$ in the norm topology of $\ell^2(\X)$, and apply the same argument (based on formula \eqref{eq5.M}) as in the end of step 3.

This completes the proof. 
 \end{proof}

\section{A hierarchy of Palm measures}

\subsection{Conditioning} 

Let, as above,  $\X$ be a countable set and $\Om:=\{0,1\}^\X$. Given a point $p\in\X$, we set 
\begin{equation}\label{eq6.A}
\X(p):=\X\setminus\{p\}, \qquad \Om(p):=\{0,1\}^{\X(p)}.
\end{equation}
Note that there is a natural projection $\Om\to\Om(p)$ sending $\om\in\Om$ to $\om\setminus\{p\}\in\Om(p)$.

Next, we decompose $\Om$ into disjoint union of two cylinder subsets,
$$
\Om=\Om_0(p)\sqcup \Om_1(p),  
$$
where
\begin{equation}\label{eq6.B}
\Om_0(p):=\{\om: p\notin\om\}, \qquad \Om_1(p):=\{\om: p\in\om\}.
\end{equation}

Both these cylinder subsets of $\Om$ are in a natural bijective correspondence with the space $\Om(p)$, induced by the projection $\Om\to\Om(p)$.

\begin{definition}\label{def6.A}
Let $M\in\P(\Om)$. We assign to $M$ two probability measures on $\Om(p)$ as follows. 

(i) The first measure, denoted by $M(p)$, is defined if $M(\Om_1(p))>0$. We restrict  
 $M$ to $\Om_1(p)$, next normalize it, and then take the pushforward under the bijection $\Om_1(p)\to\Om(p)$. The result is $M(p)$. It is called the \emph{reduced Palm measure} of $M$ at point $p$. 
 
(ii) The second measure, denoted by $\wt M(p)$, is defined if $M(\Om_1(p))<1$ (which is equivalent to $M(\Om_0(p))>0$). It is obtained by a similar procedure, only $\Om_1(p)$ should be replaced by $\Om_0(p)$. Thus, $\wt M(p)$ is simply the normalized restriction of $M$ onto the set $\Om_0(p)$.
\end{definition}

To emphasize the duality between the definitions in (i) and (ii), note that in (i) we condition upon the presence of a particle at $p$, while in (ii) we condition upon the presence of a hole.

Further, note that $M(\Om_1(p))=\rho_1(p)$, where $\rho_1$ is the first correlation function of $M$ (see \eqref{eq1.A}). Thus, the conditions $M(\Om_1(p))>0$ and $M(\Om_1(p))<1$ precisely mean that $\rho_1(p)>0$ and $\rho_1(p)<1$, respectively. If $M$ is a determinantal measure with correlation kernel $K(x,y)$, then these  conditions can be reformulated as $K(p,p)>0$ and $K(p,p)<1$, respectively.

Observe that if  $M(p)$ and $\wt M(p)$ are treated as measures on $\Om$ concentrated on the subset $\Om_0(p)\subset \Om$ (here we use the bijection $\Om(p)\leftrightarrow\Om_0(p)$), then $\wt M(p)$ is always absolutely continuous with respect to the initial measure $M$, while for $M(p)$ this may be wrong.

We are going to show that if $M\in\P(\Om)$ belongs to the class of projection kernel measures (see the definition in Subsection \ref{sect1.1}), then so is the measure $M(p)$. 

Fix a closed linear subspace $L\subset \ell^2(\X)$ and let $M[L]\in\P(\Om)$ be the corresponding determinantal measure. Recall that $M[L]$ is defined by the kernel $K^L(x,y)$, which is the matrix of $K^L$, the operator of orthogonal projection onto $L$. 

Let $e_p^\perp\subset\ell^2(\X)$ denote the codimension $1$ subspace orthogonal to the vector $e_p$. Equivalently, if elements of $\ell^2(\X)$ are viewed as functions on $\X$, then $e_p^\perp$ is the subspace of functions vanishing at $p$. There is a natural isomorphism $e_p^\perp\to\ell^2(\X(p))$.  

\begin{definition}\label{def6.B}
Given a closed linear subspace $L\subset\ell^2(\X)$, we assign to it two subspaces in $\ell^2(\X(p))$, denoted by $L(p)$ and $\wt L(p)$:

(i) $L(p)$ is  the image of the intersection $L\cap e_p^\perp$ under the isomorphism  $e_p^\perp\to \ell^2(\X(p))$.

(ii) $\wt L(p)$ is the image of $L$ under the composition of two maps: the projection $\ell^2(\X)\to e_p^\perp$ followed again by the isomorphism  $e_p^\perp\to \ell^2(\X(p))$.
\end{definition}

Note that the condition $K^L(p,p)>0$ means that $L$ is not contained in $e_p^\perp$, while $K^L(p,p)<1$ means that $L$ does not contain $\C e_p$.  

The following proposition is a variation of Shirai-Takahashi \cite[Theorem 6.5]{ST-palm} and Lyons \cite[(6.5)]{Ly-2003}.
\begin{proposition}\label{prop6.A}
Let $\X$ be a countable set, $L$ be a closed linear subspace in $\ell^2(\X)$, and $M=M[L]$ be the corresponding projection kernel measure on\/ $\Om=\{0,1\}^\X$. We fix a point $p\in\X$.

Suppose that $L$ is not contained in $e_p^\perp$, so that the reduced Palm measure $M(p)\in\P(\Om(p))$ exists. Then $M(p)=M[L(p)]$. 
\end{proposition}

\begin{proof}
Let $K$ denote the orthogonal projection $\ell^2(\X)\to L$ and $K(p)$ denote the orthogonal projection $\ell^2(\X(p))\to L(p)$.  Write the operator $K$ in the block form corresponding to the orthogonal decomposition $\ell^2(\X)=e_p^\perp\oplus\C e_p$:
$$
K=\begin{bmatrix} a & b\\ c & d\end{bmatrix}, \qquad a: e_p^\perp\to e_p^\perp, \quad b: \C e_p\to e_p^\perp, \quad c: e_p^\perp\to \C e_p, \quad d: \C e_p\to \C e_p.
$$
Note that $d$ is simply a real number, and our assumption on $L$ means that $d\ne0$. 
\smallskip

\emph{Step} 1. Let us prove that $K(p)=a-bd^{-1}c$, where we identify $e_p^\perp$ with $\ell^2(\X(p))$. Indeed, the operator $a-bd^{-1}c$ is obviously selfadjoint. The relation $K^2=K$, written in block form, reduces to a system of four relations on the blocks $a,b,c,d$. Using these relations it is readily verified that $a-bd^{-1}c$ is an idempotent. Thus, it is the operator of orthogonal projection $e_p^\perp \to L'$, where $L'$ is a subspace of $e_p^\perp$. It remains to show that $L'=L(p)$. 

Suppose $\xi\in L(p)$. Since the vector $\xi\oplus 0$ is contained in $L$, it is fixed by $K$. Writing this in terms of the block form we see that this is equivalent to two relations: $c\xi=0$ and $a\xi=\xi$. But then $(a-bd^{-1}c)\xi=\xi$, so that $\xi\in L'$.

Conversely, suppose $\xi\in L'$, which means $(a-bd^{-1}c)\xi=\xi$. Observe that $a\le1$ and $b d^{-1}c\ge0$ (because $c=b^*$). On the other hand, $(a-bd^{-1}c)\xi=\xi$ implies
$$
(a\xi,\xi)-(bd^{-1}c\xi,\xi)=(\xi,\xi).
$$
It follows that $a\xi=\xi$ and $bd^{-1}c\xi=0$. Observe also that the last relation implies $c\xi=0$. The combination of the two relations, $a\xi=\xi$ and $c\xi=0$, precisely means that $\xi\oplus 0$ is contained in $L$, so that $\xi\in L(p)$, as desired. 
\smallskip

\emph{Step} 2. Let us show that the matrix of the operator $a-bd^{-1}c$ serves as a correlation kernel for the reduced Palm measure $M(p)$. Then the proposition will be proved. 

Let $\rho_1,\rho_2,\dots$ denote the correlation functions of $M$ and $\wt\rho_1,\wt\rho_2,\dots$ denote the correlation functions of $M(p)$. By the very definition of $M(p)$, we have, for an arbitrary $n$ and arbitrary $n$-tuple $x_1,\dots,x_n$ of distinct points of $\X(p)$, 
$$
\wt\rho_n(x_1,\dots,x_n)=\frac{\rho_{n+1}(x_1,\dots,x_n,p)}{\rho_1(p)}.
$$
Denoting $x_{n+1}:=p$, we write this as
$$
\wt\rho_n(x_1,\dots,x_n)=\dfrac{\det[K(x_i,x_j]_{i,j=1}^{n+1}}{K(x_{n+1},x_{n+1})}.
$$
Write the $(n+1)\times(n+1)$-matrix in the numerator in the block form $\begin{bmatrix} \al &\be\\ \ga & \de \end{bmatrix}$, where $\al$ has format $n\times n$ and $\de$ is a number (the latter was previously denoted by $d$, and it is nonzero). Then we have 
$$
\det \begin{bmatrix} \al &\be\\ \ga & \de \end{bmatrix}=\de\cdot\det(\al-\be\de^{-1}\ga),
$$
which implies 
$$
\wt\rho_n(x_1,\dots,x_n)=\det(\al-\be\de^{-1}\ga).
$$The latter determinant coincides with the diagonal minor of the matrix of $a-bd^{-1}c$ corresponding to the $n$-tuple $x_1,\dots,x_n$. This completes the proof.  
\end{proof}

The next proposition is similar to the previous one. 

\begin{proposition}\label{prop6.A1}
Let $\X$ be a countable set, $L$ be a closed linear subspace in $\ell^2(\X)$, and $M=M[L]$ be the corresponding projection kernel measure on $\Om=\{0,1\}^\X$. We fix a point $p\in\X$.

Suppose that $L$ does not contain  $\C e_p$, so that the measure $\wt M(p)\in\P(\Om(p))$ exists. Then $\wt M(p)=M[\wt L(p)]$. 
\end{proposition}

\begin{proof}
The argument is very similar to that in Proposition \ref{prop6.A}. 
Let again $K$ denote the orthogonal projection $\ell^2(\X)\to L$, which we write in the same block form, and let $\wt K(p)$ denote the orthogonal projection $\ell^2(\X(p))\to \wt L(p)$.  
\smallskip

\emph{Step} 1. Let us prove that $\wt K(p)=a+b(1-d)^{-1}c$. The assumption that $L$ does not contain  $\C e_p$ means that $d<1$, so that $1-d\ne0$. The operator $a+b(1-d)^{-1}c$ is obviously selfadjoint. From the same system of four relations on the blocks $a,b,c,d$ one can deduce that $a+b(1-d)^{-1}c$ is an idempotent. Thus, it is the operator of orthogonal projection $e_p^\perp \to \wt L'$, where $\wt L'$ is a subspace of $e_p^\perp$. It remains to show that $\wt L'=\wt L(p)$. 

Let $\xi\in \ell^2(\X)$ be a vector. Write it in the form $\xi=\xi_1\oplus\xi_2$, where $\xi_1\in e_p^\perp$ and $\xi_2\in\C e_p$. The condition $\xi\in L$ means $K\xi=\xi$, which in turn is equivalent to the system of two relations
$$
a\xi_1+b\xi_2=\xi_1, \qquad c\xi_1+d\xi_2=\xi_2.
$$
That system is further rewritten as 
$$
\xi_2=(1-d)^{-1}c\xi_1, \qquad (a+b(1-d)^{-1}c)\xi_1=\xi_1.
$$ 
Here the first relation shows that for $\xi\in L$, the component $\xi_2$ is uniquely determined by the component $\xi_1$, while the second relation shows that $\xi_1$ belongs to $\wt L(p)$ if and only if $\xi_1$ is in the range $\wt L'$ of the projection $a+b(1-d)^{-1}c$. Therefore, that projection coincides with $\wt K(p)$. 

\emph{Step} 2. It remains to prove that the matrix of the operator $a+b(1-d)^{-1}c$ serves as a correlation kernel for the  measure $\wt M(p)$. We will use the same notation as in step 2 of the proof of Proposition \ref{prop6.A}, only now $\wt\rho_1,\wt\rho_2,\dots$ will refer to the correlation functions of $\wt M(p)$, not $M(p)$. 

We have
$$
\wt\rho_n(x_1,\dots,x_n)=\frac{\rho_n(x_1,\dots,x_n)-\rho_{n+1}(x_1,\dots,x_n,p)}{1-\rho_1(p)}.
$$
Denoting $x_{n+1}:=p$, we consider again the $(n+1)\times(n+1)$-matrix $[K(x_i,x_j)]$ and write it in the block form $\begin{bmatrix} \al &\be\\ \ga & \de \end{bmatrix}$. Then the above ratio  is written as
$$
\frac{\det\al -\det \begin{bmatrix} \al &\be\\ \ga & \de \end{bmatrix}}{1-\de}=\frac{\det \begin{bmatrix} \al &\be\\ -\ga & 1-\de\end{bmatrix}}{1-\de},
$$
where the last equality is verified by expanding the determinant on the right-hand side with respect to the last row.  Finally, the resulting expression equals $\det(\al+\be(1-\de)^{-1}\ga)$. This means that $\wt\rho_n(x_1,\dots,x_n)$ equals the corresponding diagonal minor extracted from the matrix $a+b(1-d)^{-1}c$, which completes the proof.  
\end{proof}

\subsection{Estimates for the first correlation function of $M^\z$}

We return to our concrete situation when $\Om:=\{0,1\}^{\Z'}$. Recall the notation (see \eqref{eq3.B1} and \eqref{eq3.B2})
\begin{equation*}
C(z,z'):=\frac{(z-z')\sin{\pi z}\sin(\pi z')}{\pi\sin(\pi(z-z'))}, \quad z\ne z'; \qquad C(a,a):=\left(\dfrac{\sin(\pi a)}{\pi}\right)^2, \quad a\in\R\setminus\Z.
\end{equation*}
Note that $C(z,z')>0$ for any admissible $\z$.

\begin{proposition}\label{prop6.B}
Let $\z$ be an admissible pair of parameters and $\rho^\z_1(x)$ denote the first correlation function of $M^\z$. 

{\rm(i)} As $x$ ranges from $-\infty$ to $+\infty$ along $\Z'$, the function $\rho_1^\z(x)$ strictly decreases.

{\rm(ii)} One has 
$$
\lim_{x\to+\infty}\rho^\z_1(x)=0, \qquad \lim_{x\to-\infty}\rho^\z_1(x)=1.
$$
More precisely,
\begin{equation}\label{eq6.C}
\rho^\z_1(x)\sim C(z,z') x^{-1}, \quad x\gg0, \qquad 1-\rho^\z_1(x)\sim C(z,z')|x|^{-1}, \quad x\ll0.
\end{equation}
\end{proposition}

The proposition directly implies 
\begin{corollary}\label{cor6.A}
 For any admissible pair of parameters $\z$ and any $p\in \Z'$, we have $0<\rho^\z_1(p)<1$.  In particular, the reduced Palm measure $M^\z(p)$ and the related second type measure are  well-defined for any $p\in\Z'$
\end{corollary}

\begin{proof}[Proof of Proposition \ref{prop6.B}]
(i) Suppose first that $z\ne z'$. Then we may use formula \eqref{eq2.C2} which gives us
\begin{equation*}
\rho_1^\z(x)=\frac{\sin{\pi z}\sin(\pi z')}{\pi\sin(\pi(z-z'))}\,(\psi(x+z+\tfrac12)-\psi(x+z'+\tfrac12)), \quad x\in\Z'.
\end{equation*}
Using the functional equation for the $\psi$-function (\cite[Section 1.7.1, (8)]{Er}) in the form
$$
\psi(u)-\psi(u+1)=-\frac1u
$$
we obtain
$$
\rho^\z_1(x)-\rho^\z_1(x+1)=\frac{C(z,z')}{(x+z+\tfrac12)(x+z'+\tfrac12)}.
$$
This formula also follows from \eqref{eq3.B}. Since $C(z,z')>0$, this expression is strictly positive.

For $z=z'=a$ the argument is the same: one may use \eqref{eq2.C4} instead of \eqref{eq2.C2} or apply \eqref{eq3.B}. 

(ii)  Introduce a notation: given $\om\in\Om$, we set
\begin{equation}\label{eq6.D}
X(\om)=\om\cap\Z'_+, \qquad Y(\om):=\Z'_-\setminus\om, \qquad \om^\circ: =X(\om)\sqcup \Y(\om).
\end{equation}
Note that $\om^\circ$ coincides with the symmetric difference $\om\,\triangle\,\Z'_{<0}$ and $(\om^\circ)^\circ=\om$, so that $\om\mapsto \om^\circ$ is an involutive transformation of the space $\Om$ (a `particle/hole involution', see \cite[Section 0.4]{Ols-2011}). Let $(M^\z)^\circ$ stand for the pushforward of $M^\z$ under this transformation. The asymptotic relations \eqref{eq5.B} precisely mean that the first correlation function of the measure $(M^\z)^\circ$ decays at $\pm\infty$ as $C(z,z')|x|^{-1}$, and the latter claim is precisely what is proved in \cite[Corollary 1.6]{Ols-2011}. (In \cite{Ols-2011}, the measure $(M^\z)^\circ$ is denoted by $P_{z,z'}$.)
\end{proof}

\subsection{The link between $M^\z(p)$ and $M^{(z+1,z'+1)}$.}

Let us introduce the notation which is used in the formulation of Theorem \ref{thm6.A} below. 

As above, we denote by $\Om$ the space $\{0,1\}^{\Z'}$. We fix an arbitrary point $p\in\Z'$. In accordance with \eqref{eq6.A} we set 
$$
\Z'(p):=\Z'\setminus \{p\}, \quad \Om(p):=\{0,1\}^{\Z'(p)}.
$$
Recall the definition \eqref{eq6.A} of the cylinder set $\Om_0(p)\subset\Om$. Sometimes it will be convenient to regard $\Om(p)$ as a subspace of $\Om$ by making use of the natural bijection $\Om_0(p)\leftrightarrow\Om(p)$. 

We fix  an arbitrary admissible pair of parameters $\z$. Let $M^\z$ be the corresponding gamma kernel measure  and $M^\z(p)$ be its reduced Palm measure at $p$ (it is well defined by virtue of Corollary \ref{cor6.A}). Note that the measures $M^\z$  and $M^\z(p)$ are mutually singular, see \cite{Buf-rigid}.

We also need the second type measure $\wt M(p)$ on $\Om(p)$, associated with the measure $M:=M^\zzz$ (Definition \ref{def6.A} (ii)). Its existence is also guaranteed by Corollary \ref{cor6.A}, and we denote it by $\wt M^\zzz(p)$. 
 
Introduce a function on $\Z'(p)$:
$$
a(x)=a^\z_p(x):=\dfrac{(x-p)^2}{(x+z+\tfrac12)(x+z'+\tfrac12)}, \qquad x\in\Z'(p).
$$
The function $a^\z_p$ is strictly positive and the difference $a(x)-1$ lies in $\ell^2(\Z'(p))$, as is seen from the bound
$$
a(x)=1-\frac{2p+z+z'+1}x+O\left(\frac1{x^2}\right), \qquad |x|\gg0.
$$
Therefore, by virtue of Proposition \ref{prop5.A}, the functional $\PPsi_{a,\wt M^\zzz(p)}$ is well defined, as an element of $L^1(\Om(p), \wt M^\zzz(p))$.  Note that the function $a(x)-1$ is not contained in $\ell^1(\Z'(p))$ (unless $2p+z+z'+1=0$). This is why we need to use Proposition \ref{prop5.A} (see Remark \ref{rem5.A} above).

\begin{theorem}\label{thm6.A}
For any admissible pair $\z$ and any $p\in \Z'$,  we have 
\begin{equation}\label{eq6.E}
M^\z(p)=\PPsi_{a,\wt M^{(z+1,z'+1)}}\cdot \wt M^{(z+1,z'+1)}(p).
\end{equation}
\end{theorem}

\begin{remark}\label{rem6.C}
Our definition \eqref{eq5.E} of the normalized multiplicative functionals may seem complicated because of a limit transition. However, one can prove that  $\PPsi_{a,\wt M^\zzz(p)}$ coincides, up to a constant factor, with the functional $\check\Psi_{p;z,z'}$ given by the explicit formula  \eqref{azp-funct}.
\end{remark}

\begin{remark}\label{rem6.A}
One may slightly change our viewpoint and treat $M^\z(p)$ as a measure on $\Om$ by making use of the bijection $\Om(p)\leftrightarrow\Om_0(p)$. Then, in this interpretation, the theorem implies that $M^\z$ is absolutely continuous with respect to $M^{(z+1,z'+1)}$. 
To explain this point in greater detail, note that in this paper, we only consider multiplicative functionals corresponding to  functions that never assume value zero. A slightly different equivalent  formalism is that of \cite {Bu}: in that paper, it is allowed that a function $a$ assume value zero with positive probability: in that case, if a configuration contains a particle $x$ such that $a(x)=0$, then the value of the multiplicative functional $\Psi_a$ is set to be zero at that configuration; otherwise, the definition is the same, and the same modification is adopted for normalized multiplicative functionals. If such, slightly more general, definition is adopted, then one can remove the tildes and rewrite the equation \eqref{eq6.E} in the equivalent simpler form
\begin{equation}\label{eq6.E-bis}
M^\z(p)=\PPsi_{a, M^{(z+1,z'+1)}}\cdot  M^{(z+1,z'+1)}.
\end{equation}
\end{remark}

\begin{remark}\label{rem6.B}
Theorem \ref{thm6.A} establishes a hierarchy  of Palm measures for the gamma kernels; in \cite{Bu-2017}, a similar hierarchy of Palm measures had been obtained for Bessel kernels.
\end{remark}

\begin{proof}[Proof of Theorem \ref{thm6.A}]
\emph{Step} 1. Let $\F_{z}$ stand for the space of all rational functions $f(x)$ subject to the following two conditions:

\smallskip

$\bullet$ $f(x)$ is regular at infinity and $f(\infty)=0$;

$\bullet$ the only singularities of $f(x)$ in $\C$ are simple poles contained in the set 
$$
\{-z-m-\tfrac12: m\in\Z_{\ge0}\}.
$$
(Recall that $z\in\C\setminus\Z$ because $(z,z')$ is admissible, see Definition \ref{def2.A}.)
 
Next, given $p\in\Z'$, let $\F_{z}(p)$ denote the subspace of functions $f\in\F_{z}$ vanishing at $p$. 

We claim that 
\begin{equation}\label{eq6.F}
\F_z(p)=\frac{x-p}{x+z+\tfrac12}\,\F_{z+1}.
\end{equation}

Indeed, for $n=0,1,2,\dots$, let $\F_{z,n}\subset\F_{z}$ denote the subspace of functions with all poles contained in the finite set
$$
\{-z-m-\tfrac12: 0\le m\le n\}.
$$
Evidently, $\dim\F_{z,n}=n+1$ and $\F_{z}$ is the union of the subspaces $\F_{z,n}$. 

To prove \eqref{eq6.F} it suffices to check the equality
\begin{equation*}
\F_{z}(p)\cap\F_{z,n} =\frac{x-p}{x+z+\tfrac12}\,\F_{z+1, n-1}, \qquad n\in\Z_{\ge1}.
\end{equation*}
Since $p\in\Z'$ and $z\notin\Z$, the point $p$ is not contained in the set of possible poles. It follows that the space on the right is contained in the space on the left. On the other hand, both these spaces have the same dimension $n$. This completes the proof of  \eqref{eq6.F}.  
\smallskip

\emph{Step} 2. Recall the definition \eqref{eq3.g} of the functions $g^\z_m(x)$:
\begin{equation*}
g^\z_m(x):=\frac{\sin(\pi z')\Ga(1+z-z')}{\pi}\, \frac{\Ga(x+z+\tfrac12)}{\sqrt{\Ga(x+z+\tfrac12)\Ga(x+z'+\tfrac12)}} \,\frac{\Ga(x+z'+m+\tfrac12)}{\Ga(x+z+m+\tfrac32)}.
\end{equation*}
Below $\Span\{\cdots\}$ denotes the algebraic linear span of a given set $\{\cdots\}$ of functions. 

We claim that 
\begin{multline}\label{eq6.G}
\left\{g\in\Span\{g^\z_0, g^\z_1,\dots\}: g(p)=0\right\}\\
=\frac{x-p}{\sqrt{(x+z+\tfrac12)(x+z'+\tfrac12)}}\cdot\Span\{g^\zzz_0, g^\zzz_1,\dots\}.
\end{multline}
Indeed, we will show that this relation can be reduced to \eqref{eq6.F}. 

Below we use the conventional notation for the Pochhammer symbol
$$
(t)_m:=t(t+1)\dots(t+m-1)=\frac{\Ga(t+m)}{\Ga(t)}, \quad m=0,1,2,\dots, \quad t\in\C.
$$
It is convenient to rewrite the above expression for $g^\z_m(x)$ in the form
$$
g^\z_m(x)=\ga_{z,z'}(x)\frac{(x+z'+\tfrac12)_m}{(x+z+\tfrac12)_{m+1}},
$$
where
\begin{equation*}
\ga_{z,z'}(x):=\frac{\sin(\pi z')\Ga(1+z-z')}{\pi}\,\frac{\Ga(x+z'+\tfrac12)}{\sqrt{\Ga(x+z+\tfrac12)\Ga(x+z'+\tfrac12)}}.
\end{equation*}
Next, by replacing $(z,z')$ with $(z+1,z'+1)$ we obtain
$$
g^\zzz_m(x)=\ga_{z,z'}(x)\,
\frac{x+z'+\tfrac12}{\sqrt{(x+z+\tfrac12)(x+z'+\tfrac12)}}\,\frac{(x+z'+\tfrac32)_m}{(x+z+\tfrac32)_{m+1}}.
$$

Observe now that
$$
\Span\left\{\frac{(x+z'+\tfrac12)_m}{(x+z+\tfrac12)_{m+1}}: m=0,1,2,\dots\right\}=\F_{z}
$$
and likewise
$$
\Span\left\{\frac{(x+z'+\tfrac32)_m}{(x+z+\tfrac32)_{m+1}}: m=0,1,2,\dots\right\}=\F_{z+1}.
$$

Therefore, by removing  the common factor $\ga_{z,z'}(x)$ we can reduce the desired relation \eqref{eq6.G} to the following one
$$
\F_z(p)=\frac{x-p}{\sqrt{(x+z+\tfrac12)(x+z'+\tfrac12)}}\cdot\frac{x+z'+\tfrac12}{\sqrt{(x+z+\tfrac12)(x+z'+\tfrac12)}}\cdot\F_{z+1}.
$$
Finally, after the obvious simplification we obtain just the relation \eqref{eq6.F}. This completes the proof of \eqref{eq6.G}.

\smallskip

\emph{Step} 3. 
Recall (see Section 2) that $M^\z$ possesses a correlation kernel corresponding to a selfadjoint projection operator $\K^\z$. Denote its range $\Ran(\K^\z)$ by $L^\z$.  Let us show that, in the notation of Definition \ref{def6.B},
\begin{equation}\label{eq6.H}
L^\z(p)=\phi\cdot\wt L^\zzz(p), 
\end{equation}
where we abbreviate
$$
\phi(x):=\frac{x-p}{\sqrt{(x+z+\tfrac12)(x+z'+\tfrac12)}}, \qquad x\in\Z'(p).
$$

Let us emphasize that we regard $\phi$ as a function on $\Z'(p)$, not $\Z'$. Note also that the operator of multiplication by $\phi$ is a bounded operator on $\ell^2(\Z'(p)$, and so is its inverse. 

Therefore, for the proof of \eqref{eq6.H} it suffices to exhibit dense subsets 
$$
L_*^\z(p)\subset L^\z(p), \qquad \wt L_*^\zzz(p)\subset\wt L^\zzz(p)
$$
such that 
\begin{equation}\label{eq6.I}
L_*^\z=\phi\cdot \wt L_*^\zzz(p).
\end{equation}

Let $\SSpan\{\cdots\}$ denote the closed linear span of a given set of vectors of a Hilbert space. From the results of Section \ref{sect3} we know that
\begin{equation}\label{eq6.J}
L^\z=\SSpan\{g^\z_m: m=0,1,2,\dots\}, \quad L^\zzz=\SSpan\{g^\zzz_m: m=0,1,2,\dots\}.
\end{equation}
Here, in the case of the complementary series, we assume $z<z'$.

Let us set
$$
L^\z_*(p):=\left\{f\in\Span\{g^\z_m: m=0,1,2,\dots\}, f(p)=0\right\}\bigg|_{\Z'(p)},
$$
where the symbol $\big|_{\Z'(p)}$ denotes restriction from $\Z'$ to $\Z'(p)$. Likewise, we set
$$
\wt L^\zzz_*(p):=\left\{\wt f\in\Span\{g^\zzz_m: m=0,1,2,\dots\}\right\}\bigg|_{\Z'(p)},
$$

This definition is suggested by \eqref{eq6.J}, from which we obtain that $L^\z_*(p)$ is dense in $L^\z(p)$ and, likewise, $\wt L^\zzz_*(p)$ is dense in $\wt L^\zzz(p)$. (Indeed, the latter statement is obvious while the former statement needs a little argument: here we use the fact that if $L$ is a Hilbert space, $L'\subset L$ a codimension $1$ subspace and $L_*\subset L$ a dense subspace, then $L_*\cap L'$ is dense in $L'$.)

The desired equality \eqref{eq6.I} now follows from \eqref{eq6.G}. 
\smallskip

\emph{Step} 4.
Finally we deduce the claim of the theorem from \eqref{eq6.H}. Indeed, 
from Propositions \ref{prop6.A} and \ref{prop6.B} we have
\begin{equation}\label{mzp}
M^\z(p)=M[L^\z(p)], \qquad \wt M^\zzz=M[\wt L^\zzz(p)].
\end{equation}
Substituting  the relation  \eqref{eq6.H} into Theorem \ref{thm5.A} and using  the identification \eqref{mzp}, we obtain the desired relation \eqref{eq6.E}. This completes the proof.
\end{proof}

\end{document}